\documentclass[11pt]{article}
\usepackage{lmodern}
\usepackage{makeidx}
\usepackage{amsmath,amssymb,amsthm}
\usepackage{graphicx}
\usepackage{color}
\usepackage{verbatim}
\usepackage{longtable}
\usepackage{rotating}
\usepackage{lscape}
\usepackage{multirow}
\usepackage{bigdelim}
\usepackage{csquotes}
\usepackage{enumerate}
\usepackage[english]{babel}

\definecolor{blue}{rgb}{0,0,0.9}
\definecolor{red}{rgb}{0.9,0,0}
\definecolor{green}{rgb}{0,0.9,0}

\newtheorem{assumption}{Assumption}

\newtheorem{lemma}{Lemma}

\newtheorem{theorem}{Theorem}

\def\inprod#1#2{\langle#1,#2\rangle}

\def\cS{{\cal S}}

\newcommand{\bQ}{\boldsymbol Q}
\newcommand{\bL}{\boldsymbol L}
\newcommand{\bv}{\boldsymbol v}

\newcommand{\R}{\mathbb R}
\newcommand{\N}{\mathbb N}
\newcommand{\K}{\ensuremath{\mathbf K}}
\newcommand{\A}{\mathbf A}
\newcommand{\B}{\mathbf B}
\newcommand{\M}{\mathbf M}
\newcommand{\y}{\boldsymbol y}
\newcommand{\x}{\mathbf x}

\newcommand{\n}{\mathbf n}
\newcommand{\oo}{\mathbf o}

\newcommand{\ms}{\mathcal{S}}
\newcommand{\blambda}{\boldsymbol \lambda}
\newcommand{\btheta}{\boldsymbol \theta}

\newcommand{\dmax}{\ensuremath{d_{\text{max}}}}
\newcommand{\fl}{\ensuremath{f^\ell}}
\newcommand{\cf}{\ensuremath{\mathrm{f}}}

\newcommand{\supp}{\mathrm{supp}}

\title{Sparse-BSOS: a bounded degree SOS hierarchy for large scale polynomial optimization with sparsity
\thanks{
The work of the first author is partially supported by a PGMO grant from {\it Fondation Math{\'e}matique Jacques Hadamard} and a grant from the {\it ERC council} for the Taming project (ERC-Advanced Grant \#666981 TAMING).
 }
}

\author{
Tillmann Weisser\thanks{LAAS-CNRS, University of
Toulouse, LAAS, 7 Avenue du Colonel Roche, 31031 Toulouse c\'{e}dex 4, France ({\tt tweisser@laas.fr}).}
\
Jean B. Lasserre\thanks{LAAS-CNRS and Institute of Mathematics, University of
Toulouse, LAAS, 7 Avenue du Colonel Roche, 31031 Toulouse c\'{e}dex 4, France ({\tt lasserre@laas.fr}).}
 \
 and Kim-Chuan Toh\thanks{Department of Mathematics, National University of Singapore, 
        10 Lower Kent Ridge Road, Singapore 119076 ({\tt mattohkc@nus.edu.sg}).} 
}

\begin{document}
\maketitle
\begin{abstract}
We provide a sparse version of the bounded degree SOS hierarchy BSOS \cite{bsos} for polynomial optimization problems. It permits to treat large scale problems which satisfy a structured sparsity pattern. When the sparsity pattern satisfies the running intersection property this Sparse-BSOS hierarchy of semidefinite programs (with semidefinite constraints of fixed size) converges to the global optimum of the original problem. Moreover, for the class of SOS-convex problems, finite convergence takes place at the first step of the hierarchy, just as in the dense version. 
\end{abstract}


\section{Introduction}

We consider the polynomial optimization problem:
\begin{equation}
\label{def-pb}
\quad f^*:=\displaystyle\min_\x\:\{f(\x)\::\: \x\in\K\:\} \tag{P}
\end{equation}
where $f\in\R[\x]$ is a polynomial and $\K\subset\R^n$ is the basic semi-algebraic set
\begin{equation}
\label{setk}
\K:=\{\x\in\R^n: g_j(\x)\geq0,j=1,\ldots,m\},
\end{equation}
for some polynomials $g_j\in\R[\x]$, $j=1,\ldots,m$. In \cite{bsos} Lasserre et al. have provided BSOS, a new hierarchy of semidefinite programs $(\bQ^k_d)$ indexed by $d\in\N$ and parametrized by $k\in\N$ (fixed), whose associated (monotone non-decreasing) sequence of optimal values $(\rho^k_d)_{d\in\N}$ converges to $f^*$ as $d\to\infty$, i.e., $\rho^k_d\to f^*$ as  $d\to\infty$. 

One distinguishing feature of the BSOS hierarchy (when compared to the LP-hierarchy defined in \cite{Lasserre02a}) is {\it finite convergence} for an important class of convex problems. That is, when $f,-g_j$ are SOS-convex polynomials of degree bounded by $2k$, then the first semidefinite relaxation of the hierarchy $(\bQ^k_d)_{d\in\N}$, is exact, i.e., $\rho^k_1=f^*$. (In contrast the LP-hierarchy cannot converges in finitely many steps for such convex problems).

The BSOS hierarchy has also two other important distinguishing features, now when compared to the standard SOS hierarchy defined in \cite{lasserre-siopt,lasserre-imperial,lasserre-cambridge} (let us call it PUT as it is based on Putinar's Positivstellensatz). 
\begin{itemize}
\item For each semidefinite relaxation $\bQ^k_d$, $d\in\N$, the size of the semidefinite constraint is $O(n^k)$, hence fixed and controlled by the parameter $k$ (fixed and chosen by the user). With $k=0$ one retrieves the LP-hierarchy based on a positivity certificate due to Stengle; see \cite{bsos} and \cite{Lasserre02a} for more details. 
\item For the important class of quadratic/quadratically constrained problems, the first relaxation of the BSOS hierarchy is at least as good as the first relaxation of the PUT hierarchy.
\end{itemize}
 
In this paper we introduce a sparse version of the BSOS hierarchy to help solve large scale polynomial optimization problems that exhibit some structured sparsity pattern. This hierarchy has exactly the same advantages as BSOS, now when compared with the sparse version of the standard SOS hierarchy introduced in Waki et al. \cite{waki}.

\subsection*{Motivation}
In general the standard SOS hierarchy PUT \cite{lasserre-siopt} is considered very efficient (and hard to beat) when it can be implemented. For instance, PUT's first relaxation solves at optimality several small size instances of the Optimal Power Flow problem (OPF), an important problem in management of energy networks (essentially a quadratic/quadratically constrained optimization problem); see e.g. \cite{molzahn2}. But even for relatively small size OPFs, PUT's second relaxation cannot be implemented because the size of some matrices required to be positive semidefinite (psd) is too large for state-of-the art semidefinite solvers.  

This was an important motivation for introducing the BSOS hierarchy \cite{bsos}: That is, to provide an alternative to PUT, especially in cases where PUT's second relaxation cannot be implemented. In particular, in {\it all} relaxations of the BSOS hierarchy (with parameter $k=1$) the size of the matrix required to be psd is only $O(n)$ (in contrast to $O(n^2)$ for PUT's second relaxation). Thus for quadratic/quadratically constrained problems where PUT's second relaxation cannot be implemented, the second and even higher order relaxations of BSOS can be implemented and provide better lower bounds.
\footnote{Recently Marandi et al. \cite{marandi} have shown  that BSOS provides better bounds than the first relaxation of
PUT for the Pooling Problem - another instance of quadratic/quadratically constrained problems. However their
experiments also show that BSOS (i) does not always solve these difficult problems to optimality
at a low level \enquote{$d$} relaxation and (ii) does not scale well.} 
When they are strictly less than the optimal value these lower bounds might still be useful as they can be exploited in some other procedure.

The motivation for introducing the sparse version of BSOS is exactly the same as for BSOS, but now for large scale polynomial optimization problems that exhibit some sparsity pattern. For such problems the sparse version of the SOS hierarchy introduced in Waki et al. \cite{waki} (that we call \enquote{Sparse-PUT} in the sequel) is in general very efficient, in particular when its second relaxation can be implemented. So there is a need to provide an alternative to the latter, especially in cases where its second relaxation cannot be implemented -- for instance for some large scale OPF problems (where indeed the second relaxation cannot be implemented, at least in its initial form), e.g as described in \cite{molzahn}.

\subsection*{Contribution}

Even though in the BSOS hierarchy \cite{bsos} the size $O(n^k)$ of the semidefinite constraint of $\bQ^k_d$ is fixed for all $d$ and permits to handle problems $(P)$ of size larger than with the standard SOS-hierarchy, its application is still limited to problems of relatively modest size (say medium size problems). Our contribution is to provide a {\it sparse} version \enquote{Sparse-BSOS} of the BSOS hierarchy which permits to handle large size problems $(P)$ that satisfy some (structured) sparsity pattern.
The Sparse-BSOS hierarchy is to its dense version BSOS what the Sparse-PUT hierarchy is to the standard SOS-hierarchy PUT.
Again as in the dense case, a distinguishing feature of the Sparse-BSOS hierarchy (and in contrast to Sparse-PUT) is that the size of the resulting semidefinite constraints is {\it fixed in advance} at the user convenience and does {\it not} depend on the rank in the hierarchy. However, such an extension is not straightforward because in contrast to Putinar's SOS-based certificate \cite{Putinar93} (where the $g_j$'s appear separately), the positivity certificate used in the dense BSOS algorithm \cite{bsos} potentially mixes all polynomials  $g_j$ that define $\K$, that is, 
if $f$ is positive on $\K$ then
\begin{equation}
\label{bsos-certificate}
f\,=\,\sigma+\sum_{\alpha,\beta\in(\N_0)^m}c_{\alpha\beta}\prod_{j=1}^m g_j^{\alpha_j}(1-g_j)^{\beta_j},
\end{equation}
for some SOS polynomial $\sigma$ and positive scalar weights $c_{\alpha\beta}$. Therefore in principle the sparsity as defined in \cite{waki} may be destroyed in $\sigma$ and in the products $\prod_jg_j^{\alpha_j}(1-g_j)^{\beta_j}$. In fact, one contribution of this paper is to provide a specialized sparse version of \eqref{bsos-certificate}. In particular, we prove that if the sparsity pattern satisfies the so-called {\it Running Intersection Property} (RIP) then the Sparse-BSOS hierarchy also converges to the global optimum $f^*$ of $(P)$. A sufficient rank-condition also permits to detect finite convergence. Last but not least, we also prove that the Sparse-BSOS hierarchy preserves two distinguishing features of the dense BSOS hierarchy.
Namely:
\begin{itemize}
\item Finite convergence at the first step of the hierarchy for the class of SOS-convex problems. (Recall that the standard LP hierarchy cannot converge in finitely many steps for such problems \cite{Lasserre02a,lasserre-cambridge}.)
\item For quadratic/quadratically constrained problems (with a sparsity pattern), the first relaxation of the Sparse-BSOS hierarchy is at least as good as that of Sparse-PUT.
Hence when Sparse-PUT's second relaxation cannot be implemented the Sparse-BSOS hierarchy (with parameter $k=1$) will provide better lower bounds.
\end{itemize}
This also suggests that the Sparse-BSOS relaxations could be useful in Branch \& Bound algorithms for solving large scale Mixed Integer Non Linear programs (MINLP). Indeed for such algorithms the quality of lower bounds computed at each node of the search tree is crucial for the overall computational efficiency.

\paragraph{The sparsity pattern.} Roughly speaking we say that problem \eqref{def-pb} satisfies a structured sparsity pattern, if the set $I_0:=\{1,\ldots,n\}$ of all (indices of) variables is some union $\cup_{\ell=1}^p I_k$ of smaller blocks of (indices of) variables $I_\ell$ such that each monomial of the objective function only consists of variables in one of the blocks. In addition, each polynomial  $g_j$ in the definition \eqref{setk} of the feasible set, is also a polynomial only in variables of one of the blocks. Of course the blocks $I_1,\ldots,I_p$ may overlap, i.e., variables may appear in several blocks. Together with the maximum degree appearing in the data of (P), the number and size of the blocks $I_1,\ldots,I_p$ as well as the size of their overlaps, are the characteristics of the sparsity pattern which have the strongest influence on the performance of our algorithm.

\paragraph{Computational experiments.} We have tested the Sparse-BSOS hierarchy against both its dense version BSOS and Sparse-PUT \cite{waki}. To compare Sparse-BSOS and BSOS we consider problems of small and medium size ($\leq 100$ variables) and different sparsity patterns. Also we show that Sparse-BSOS can solve sparse large scale problems with up to $3000$ variables which by far is out of the scope of the dense version in a reasonable time on a standard lap-top.\footnote{The numerical experiments were run on a standard lap-top from the year 2015. For a more detailed description see page \pageref{computer}.}

To compare Sparse-BSOS and Sparse-PUT we have considered several problems from the literature in non-linear optimization and random medium size non-linear problems. For several such problems the first or second relaxation of the Sparse-PUT hierarchy cannot be solved whereas the first Sparse-BSOS relaxations $d=1,2\ldots$ with a small parameter $k=1$ or $2$ can\footnote{If in the description \eqref{def-pb} some degree \enquote{s} is strictly larger than $2$ then for Sparse-PUT's first relaxation, the size of the largest positive semidefinite variable is already at least $O(\kappa^{\lceil s/2\rceil})$ where $\kappa$ is the size of the largest clique in some graph associated with the problem; if $\kappa$ is not small it can be prohibitive. In contrast, for all the Sparse-BSOS relaxations (with parameter $k=1$), this largest size is only $O(\kappa)$.}. Therefore in such cases a good approximation (or even the optimal value) can be obtained by the Sparse-BSOS hierarchy. 
In our numerical experiments the problems were chosen in such a way that only the first relaxation of Sparse-PUT can be implemented, and indeed for such cases the Sparse-BSOS relaxations provide better lower bounds (and sometimes even the optimal value).

\section{Preliminaries}

\subsection{Notation and definitions}
Let $\R[\x]$ be the ring of polynomials in the variables $\x=(x_1,\ldots,x_n)$. Denote by $\R[\x]_d\subset\R[\x]$ the vector space of polynomials of degree at most $d$, which has dimension $s(d):=\binom{n+d}{d}$, with e.g., the usual canonical basis $(\x^\gamma)_{\gamma\in\N^n_d}$ of monomials, where $\N^n_d := \{\gamma\in(\N_0)^n \,:\, |\gamma|\leq d\}$, $\N_0$ is the set of natural numbers including $0$ and $|\gamma| := \sum_{i=1}^n\gamma_i$. Also, denote by $\Sigma[\x]\subset\R[\x]$ (resp. $\Sigma[\x]_d\subset\R[\x]_{2d}$) the cone of sums of squares (SOS) polynomials (resp. SOS polynomials of degree at most $2d$). 
If $f\in\R[\x]_d$, we write $f(\x)=\sum_{\gamma\in\N^n_d}\cf_\gamma \x^\gamma$ in the canonical basis and denote by $\boldsymbol{\cf}=(\cf_\gamma)_\gamma\in\R^{s(d)}$ its vector of coefficients. Finally, let $\ms^n$ denote the space of $n\times n$ real symmetric matrices, with inner product $\langle \A,\B\rangle ={\rm trace}\,\A\B$. We use the notation $\A\succeq 0$ to denote that $\A$ is positive semidefinite (psd). With $g_0:=1$, the quadratic module $Q(g_1,\ldots,g_m)\subset\R[\x]$ generated by polynomials $g_1,\ldots,g_m$, is defined by
\[
Q(g_1,\ldots,g_m)\,:=\,\left\{\sum_{j=0}^m\sigma_j\,g_j\::\:\sigma_j\in\Sigma[\x]\,\right\}.
\]
With a real sequence $\y=(y_{\gamma})_{\gamma\in\N^n_d}$, one may associate the linear functional $L_{\y}:\R[\x]_d\to\R$ defined by
\[f\,\left(=\sum_\gamma \cf_\gamma\,\x^\gamma\right)\quad \mapsto L_{\y}(f)\,:=\,\sum_\gamma \cf_\gamma\,y_\gamma,\]
which is called the Riesz functional.
If $d=2a$ denote by $\M_a(\y)$ the moment matrix associated with $\y$. It is a real symmetric matrix with rows and columns indexed in the basis of monomials $(\x^\gamma)_{\gamma\in \N^n_{a}}$, and with entries
\[
\M_a(\y)(\alpha,\beta)\,:=\,L_{\y}(\x^{\alpha+\beta})\,=\,y_{\alpha+\beta},\qquad \forall\,\alpha,\beta\in\N^n_{a}.
\]
If $\y=(y_{\gamma})_{\gamma\in\N^n}$ is the sequence of moments of some Borel measure $\mu$ on $\R^n$ then $\M_a(\y)\succeq0$ for all $a\in\N$. However the converse is not true in general and it is related to the well-known fact that there are positive polynomials that are not sums of squares.
For more details the interested reader is referred to e.g. \cite[Chapter 3]{lasserre-imperial}.

A polynomial $f\in\R[\x]$ is said to be SOS-convex if its Hessian matrix $\x\mapsto \nabla^2f(\x)$ is an SOS matrix-polynomial, that is, $\nabla^2f=\bL\,\bL^T$ for some real matrix polynomial $\bL\in\R[\x]^{n\times a}$ (for some integer $a$). In particular, for SOS-convex polynomials and sequences $\y$ with positive semidefinite moment matrix $\M_a(\y)\succeq0$, a Jensen-type inequality is valid:

\begin{lemma}
\label{lemma-jensen}
Let $f\in\R[\x]_{2a}$ be SOS-convex and let $\y=(y_{\gamma})_{\gamma\in\N^n_{2a}}$ be such that
$y_0=1$ and $\M_a(\y)\succeq0$. Then
\[L_{\y}(f)\,\geq\,f(L_{\y}(\x)),\quad\mbox{with}\quad L_{\y}(\x)\,:=\,(L_{\y}(x_1),\ldots,L_{\y}(x_n)).\]
\end{lemma}
For a proof see Theorem 13.21, p. 209 in \cite{lasserre-cambridge}.

\subsection{A sparsity pattern}
Given $I\subset\{1,\ldots,n\}$ denote by $\R[\x;I]$ the ring of polynomials in the variables $\{x_i: i\in I\}$, which we understand as a subring of $\R[\x]$. Hence, a polynomial $g\in\R[\x;I]$ canonically induces two polynomial functions $g:\R^{\#I}\to \R$ and $g:\R^n\to \R$, where $\#I$ is the number of elements in $I$.

\begin{assumption}[Sparsity pattern]\label{assumption1}
There exists $p\in\N$ and subsets $I_\ell\subseteq\{1,\ldots,n\}$ and $J_\ell\subseteq\{1,\ldots,m\}$ for all $\ell \in \{1,\ldots,p\}$ such that
\begin{itemize}
\item $f=\sum_{\ell=1}^p f^\ell$, for some $f^1,\ldots, f^p$ such that $f^\ell\in\R[\x;I_\ell]$, $\ell\in\{1,\ldots,p\}$,
\item $g_j\in \R[\x;I_\ell]$ for all $j\in J_\ell$ and $\ell\in\{1,\ldots,p\}$,
\item $\bigcup_{\ell=1}^p I_\ell = \{1,\ldots,n\}$,
\item $\bigcup_{\ell=1}^p J_\ell = \{1,\ldots,m\}$;
\item for all $\ell = 1,\ldots,p-1$ there is an $s\leq \ell$ such that $\left(I_{\ell+1}\cap\bigcup_{r=1}^\ell I_r\right) \subseteq I_s$  (Running Intersection Property).
\end{itemize}
\end{assumption}

\paragraph{On the RIP.} The RIP can be understood in the following framework. Suppose that we are given probability measures $(\mu_\ell)_{\ell=1,\ldots,p}$, where each $\mu_\ell$ only {\it sees} the variables $\{x_i:i\in I_\ell\}$, and the $\mu_\ell$'s are consistent in the sense that if $I_{\ell,k}:=I_\ell\cap I_k\neq\emptyset$ then $\mu_{\ell,k}=\mu_{k,\ell}$, where $\mu_{\ell,k}$ (resp. $\mu_{k,\ell}$) is the marginal of $\mu_\ell$ (resp. $\mu_k$) with respect to the variables $\{x_i:i\in I_{\ell,k}\}$. If the RIP property holds then from the $\mu_\ell$'s one is able to build up a probability measure $\phi$ that sees {\it all} variables $x_1,\ldots,x_n$, and such that for every $\ell=1,\ldots,p$, the marginal  $\phi_\ell$ of $\phi$  with respect to the variables $\{x_i:i\in I_\ell\}$ is exactly $\mu_\ell$. Put differently, the \enquote{local} information $(\mu_\ell)$ is part of a \enquote{consistent and global} information $\phi$, without knowing $\phi$. In this set up 
the RIP condition appears naturally when one tries to build up $\phi$ from the $\mu_\ell$'s; see e.g. the proof in Lasserre \cite{lasserre-sparse}.

Note that if the sets $I_1,\ldots,I_p$ satisfy all requirements of Assumption \ref{assumption1} except the Running Intersection Property (RIP) it is always possible to enlarge the $I_\ell$ until the RIP is satisfied. In the worst case however $I_\ell = \{1,\ldots,n\}$ which is not interesting. In \cite{waki} no assumption on the sparsity is made. Instead Waki et al. provide an algorithm to create a sparsity pattern satisfying Assumption \ref{assumption1} from any POP. They consider a graph with nodes $\{1,\ldots,n\}$ which has an edge $(i,j)$ if the objective function has a monomial involving $x_i$ and $x_j$ or there is a constraint involving $x_i$ and $x_j$. If this graph is chordal, the sets $I_\ell$ correspond to the maximum cliques of this graph (and hence we sometimes call the $I_\ell$ \enquote{cliques}). If the graph is not chordal it is replaced by its chordal extension.

Finally and importantly, even if the RIP property does not hold, the Sparse-PUT or Sparse-BSOS 
hierarchies can still be implemented and provide valid lower bounds on the global optimum $f^*$.
However convergence of the lower bounds to $f^*$ is not guaranteed any more.

\subsection{A preliminary result}
For the uninitiated reader we start this section by citing two important Positivstellens\"atze. For polynomials 
$g_1,\ldots,g_m\in\R[\x]$ and $\alpha,\beta\in (\N_0)^m$,  let $h_{\alpha\beta}\in\R[\x]$ be the polynomial:
\[
\x\mapsto\quad h_{\alpha\beta}(\x)\,:=\,\prod_{j=1}^m g_j(\x)^{\alpha_j}(1-g_j(\x))^{\beta_j},\qquad\x\in\R^n.
\]

\begin{lemma}[Krivine/Stengle/Vasilescu/Handelmann Positivstellensatz \cite{Krivine64a,Stengle74,vasilescu,Handelman88}]\label{stengle}
Let $f$, $g_1,\ldots,g_m\in\R[\x]$ and $\K=\,\{\x\in\R^n: 0\leq g_j(\x)\leq 1,\quad j=1,\ldots,m\,\}$ be compact. If the polynomials $1$ and $(g_j)_{j=1\ldots,m}$ generate $\R[\x]$ as an $\R$-algebra and $f$ is (strictly) positive on $\K$, then there exist finitely many positive weights $c_{\alpha\beta}$ such that:
\[
f = \sum_{\alpha,\beta\in(\N_0)^m}c_{\alpha\beta}\,h_{\alpha\beta}.
\]
\end{lemma}

We next provide a sparse version of this Positivstellensatz which to the best of our knowledge is new. Before
we recall the sparse version of Putinar's Positivstellensatz \cite{Putinar93} proved in \cite{lasserre-sparse}.

\begin{lemma}[Sparse Putinar Positivstellensatz \cite{kojima}]\label{sparse-putinar}
Let $f,g_1,\ldots,g_m\in\R[\x]$ satisfy Assumption \ref{assumption1} and let the polynomials $1- M_\ell^{-1}\sum_{i\in I_\ell}x_i^2$ for some positive numbers $M_\ell$ be among the $g_j$. If $f$ is (strictly) positive on $\K =\{ \x\in\R^n: g_j(\x) \geq 0,j=1,\ldots,m \},$ then $f=\sum_{\ell=1}^p \fl$ for some polynomials $\fl\in\R[\x;I_\ell]$, $\ell=1,\ldots,p$, and 
\[
\fl \in \left\{\sigma^\ell_0 + \sum_{j\in J_\ell}\sigma^\ell_j\,g_j\::\:\sigma^\ell_0,\sigma^\ell_j\in\Sigma[\x,I_\ell]\,\right\}.
\]
\end{lemma}

From now on, we assume that \K\ is described by polynomials $g_1,\ldots,g_m$ such that 
\begin{equation}\label{setkk}
\K=\,\{\x\in\R^{n}:\:0\leq g_j(\x)\leq 1,\quad j=1,\ldots,m\,\}.
\end{equation}
Note that this is not a restriction when  \K\ defined in \eqref{def-pb} is compact, as the constraint polynomials can be scaled by a positive factor without adding or losing information.
 
Let Assumption \ref{assumption1} hold, define $n_\ell:=\vert I_\ell\vert$, $m_\ell:=\vert J_\ell\vert$, and let $\K_\ell\subset\R^{n_\ell}$, $\ell=1,\ldots,p$,  be the sets:
\begin{equation}
\label{setskell}
\K_\ell\,:=\,\{\x\in\R^{n_\ell}:\:0\leq g_j(\x)\leq 1,\quad j\in J_\ell\,\},\qquad \ell=1,\ldots,p.
\end{equation}

Define $\pi_{\ell}:\R^n\to \R^{n_\ell}, \x \mapsto (x_i)_{i\in I_{\ell}}$. Then
\begin{equation}\label{eq:KKl}
 \K = \{\x\in\R^n\,:\, \pi_{\ell}(\x)\in \K_\ell\mbox{ for all }\ell=1,\ldots,p\}.
\end{equation}

\begin{assumption}\label{assumption2}
\begin{enumerate}[(a)]
\item The sets $\K_\ell$ defined in \eqref{setskell} are compact and the polynomials $1$ and $(g_j)_{j\in J_\ell}$ generate $\R[\x;I_\ell]$ as an $\R$-algebra for each $\ell=1,\ldots,p$.
\item For each $\ell$ there exists $M_\ell > 0$ and $j\in J_\ell$ such that $g_j = 1- M_\ell^{-1}\sum_{i\in I_\ell}x_i^2$.
\end{enumerate}
\end{assumption}

Note that if Assumption \ref{assumption2}(a) holds then Assumption \ref{assumption2}(b) can be satisfied easily. Indeed, since $\K_\ell$ is assumed to be compact, the polynomial $\x\mapsto \sum_{i\in I_\ell}x_i^2$ attains its maximum $M_\ell$ on $\K_\ell$. Defining $\x\mapsto g^\ell_+(\x):=1 - M_\ell^{-1}\sum_{i\in I_\ell}x_i^2$ and adding the redundant constraint $0 \leq g^\ell_+(\x) \leq 1$ to the description of $\K_\ell$ for each $\ell=1,\ldots,p$, Assumption \ref{assumption2}(b) is satisfied.

Let $N^\ell:=\{(\alpha,\beta)\in(\N_0)^{2m}\,:\,\supp(\alpha)\cup\supp(\beta)\subseteq J_\ell \}$, where $\supp(\alpha):=\{j \in \{1,\ldots,m\}\,:\, \alpha_j\neq 0\}$.

\begin{theorem}[Sparse Krivine-Stengle Positivstellensatz]
\label{sparse-stengle}
Let $f,g_1,\ldots,g_m\in\R[\x]$ satisfy Assumption \ref{assumption1} and \ref{assumption2}(a). If $f$ is (strictly) positive on $\K$ then $f=\sum_{\ell=1}^p \fl$ for some polynomials $\fl\in\R[\x;I_\ell]$, $\ell=1,\ldots,p$, and there exists finitely many positive weights $c^\ell_{\alpha\beta}$ such that

\begin{equation}
\label{sparse-stengle-1}
\fl\,=\,\sum_{(\alpha,\beta)\in N^{\ell}}c^\ell_{\alpha\beta}\,h_{\alpha\beta},\qquad \ell=1,\ldots,p.
\end{equation}

\end{theorem}
Note that for each $\ell$ the representation \eqref{sparse-stengle-1} only involves $g_j\in \R[\x;I_\ell]$ since the corresponding exponents in the definition of $h_{\alpha\beta}$ are $0$ if $j \notin J_\ell$.
\begin{proof}
As $f$ is positive on $\K$ there exist $\varepsilon>0$ such that $f-\varepsilon>0$ on $\K$. As remarked just after  Assumption \ref{assumption2}, we can add a redundant constraint $0 \leq g^\ell_+(\x)\leq 1$, with $g^\ell_+\in\R[\x;I_\ell]$, to the description of each of the $\K_\ell$ so as to satisfy Assumption \ref{assumption2}(b). Hence we can apply Lemma \ref{sparse-putinar} to obtain the representation 

\[
f-\varepsilon=\sum_{\ell=1}^p \left(\underbrace{\sigma^\ell_0+  \sigma^\ell_+g^\ell_+ + \sum_{j\in J_\ell}\sigma^\ell_j\,g_j}_{\in\R[\x;I_\ell]\mbox{ and }\geq 0\mbox{ on }
\K_\ell}\right),
\]
for some SOS polynomials $\sigma^\ell_j,\sigma^\ell_+$. Next, let
\[
\fl\,:=\,\varepsilon/p+\sigma^\ell_0+ \sigma^\ell_+g^\ell_+ +\sum_{j\in J_\ell}\sigma^\ell_j\,g_j,\qquad \ell=1,\ldots,p.
\]
Notice that $f=\sum_{\ell=1}^p\fl$ and each $\fl\in\R[\x;I_\ell]$ is strictly positive on $\K_\ell$, $\ell=1,\ldots,p$. Removing the redundant constraints $0 \leq g^\ell_+\leq 1$ from the description of $\K_\ell$ again does not change the fact that $\fl$ is strictly positive on $\K_\ell$. Furthermore Assumption \ref{assumption2}(a) still holds. Hence, we can apply Lemma \ref{stengle} to each of the $f^\ell$, which yields \eqref{sparse-stengle-1} for each $\ell=1,\ldots,p$.

\end{proof}
		
\section{Main result}

\subsection{The Sparse Bounded-Degree SOS-hierarchy (Sparse-BSOS)}
Consider problem \eqref{def-pb} and let Assumption \ref{assumption1} \& \ref{assumption2} hold. For $d\in\N$ let $N^\ell_d:=\{(\alpha,\beta)\in N^\ell:\sum_j\left(\alpha_j+\beta_j\right)\leq d \}$, which is of size $\binom{2 m_\ell+d}{d}$. Let $k\in\N$ be fixed and define 
\[
\dmax := \max\{\deg(f),2k,d\max_j\{\deg(g_j)\}\}.
\] 
We consider the family of optimization problems indexed by $d\in\N$:

\begin{equation}\label{aux_{1}}
\begin{split}
q_d^k := \sup_{\scriptstyle\begin{array}{cc}t,\blambda^1,\ldots,\blambda^p,\\ f^1,\ldots, f^p\end{array}} \left\{\,  t\;:\;\right.
&\fl-\displaystyle\sum_{(\alpha,\beta)\in N^{\ell}_d}
\lambda^\ell_{\alpha\beta} h_{\alpha\beta}\in\Sigma[\x;I_\ell]_k,\quad \ell=1,\ldots,p,\\
& f\,-t\,=\,\sum_{\ell=1}^p\fl,\; \left. \blambda^\ell\in\R^{|N^\ell_d|}_+,\: t\in\R,\:\fl\in\R[\x;I_\ell]_{\dmax} \right\},
\end{split}
\end{equation}
where the scalars $\lambda^\ell_{\alpha\beta}$ are the entries of the vector ${\boldsymbol \lambda}^\ell\in\R^{|N^\ell_d|}_+ $.

Observe that when $k$ is fixed, then for each $d\in\N$, computing $q^k_d$ in \eqref{aux_{1}} reduces to  solving a semidefinite program and $q^k_{d+1}\geq q^k_d$ for all $d\in\N$. Therefore \eqref{aux_{1}} defines a {\it hierarchy} of semidefinite programs of which the associated sequence of optimal values is monotone non decreasing. Let us call it the Sparse-BSOS hierarchy, as it is the {\it sparse version} of the BSOS hierarchy \cite{bsos}.

\label{SDP-formulation}
For practical implementation of \eqref{aux_{1}} there are at least two possibilities depending on how the polynomial identities in \eqref{aux_{1}} are treated in the resulting semidefinite program. To state that two polynomials $p,q\in\R[\x]_d$ are identical one can either:

-  {\it equate their coefficients} (e.g. in the monomial basis, i.e., $\mathrm{p}_\gamma=\mathrm{q}_\gamma$ for all $\gamma\in\N^n_d$), or 

- {\it equate their values} (i.e an implementation by \textit{sampling}) on $\binom{n+d}{d}$ generic points (e.g. randomly generated on  the box $[-1,1]^n$). 

Both strategies have drawbacks: To equate coefficients one has to take powers of the polynomials $g_j$ and $(1-g_j)$ which leads to an ill-conditioning of the coefficients of the polynomials $h_{\alpha\beta}$ (in the monomial basis) as some of them are multiplied by binomial coefficients which become large quickly when the relaxation order $d$ increases. On the other hand, when equating values the resulting linear system may become ill-conditioned because (depending on the points of evaluation and the $g_j$) the constraints may be nearly linear dependent.
The authors of \cite{bsos} chose point evaluation for the implementation of BSOS because the SDP solver SDPT3\footnote{In both \cite{bsos} and this paper SDPT3 \cite{TTT99,TTT03} is used as SDP solver.} is able to exploit the structure of the SDP generated in that way and hence problems with psd variables of larger size can be solved. However, this feature cannot be used in our case and so in Sparse-BSOS we have implemented the equality constraints of \eqref{aux_{1}} by comparing coefficients.
 
Indeed, equating coefficients is reasonable in the present context because we expect the number of variables $n_\ell$ in each block to be rather small. The drawback of this choice is that the resulting relaxations with high order $d$ can become time consuming (and even ill-conditioned as explained above).

Define $\mathcal{I}_\ell^d:=\{\gamma\in\N^n_d\, :\,\supp(\gamma)\in I_\ell \}$, $\Gamma^d := \{\gamma\;:\;\gamma \in \mathcal{I}_\ell^d,\mbox{for some } \ell\in \{1,\ldots,p\} \}$ and let $\boldsymbol{0}$ be the all zero vector of size $n$. Then for $k$ fixed and for each $d$, we get

\begin{equation}\label{primal}
\begin{split}
q_d^k =\sup_{\substack{t,Q^1,\ldots,Q^p,\\ \blambda^1,\ldots,\blambda^p,\\ \cf^1,\ldots, \cf^p}}\left\{\,t\quad\right.\text{s.t. } 
& \sum_{\ell:\:\gamma\in \mathcal{I}_{\ell}^{\dmax}}\cf^\ell_{\gamma} = \cf_{\gamma},\quad\forall \gamma\in\Gamma^{\dmax}\backslash\{\boldsymbol{0}\},\quad \sum_{\ell=1}^p\cf^\ell_{\boldsymbol{0}} = \cf_{\boldsymbol{0}}-t \\
& \cf^\ell_\gamma  -\sum_{(\alpha,\beta)\in N^{\ell}_d} \lambda^\ell_{\alpha\beta}\,(h_{\alpha\beta})_\gamma - \inprod{Q^\ell}{\left(\bv_k^\ell (\bv_k^\ell)^T\right)}_\gamma = 0,\quad\\ 
& \hspace*{5cm}\forall \gamma\in\mathcal{I}_\ell^{\dmax},\quad \ell=1,\ldots,p, \\ 
& \left. Q^\ell \in \cS^{s(\ell,k)}_+,\; \blambda^\ell\in\R^{|N^l_d|}_+ ,\; \boldsymbol{\cf}^\ell\in \R^{s(\ell,\dmax)},\;\ell=1,\ldots,p,\;t\in\R \right\},
\end{split}
\end{equation}
where $s(\ell,k) := \binom{n_\ell+k}{k}$, and $\bv_k^\ell$ is the vector of the canonical (monomial) basis of the vector space $\R[\x;I_\ell]_k$. Here we use the convention that the coefficient $\mathrm{q}_\gamma$ of a polynomial $q$ is $0$ if $|\gamma| > \deg(q)$. For a matrix polynomial $\mathbf{q} = (q_{ij})_{1\leq i,j\leq s}\in\R[\x]^{s\times s}$ the coefficient $\mathbf{q}_\gamma$ is the matrix $((q_{ij})_\gamma)_{1\leq i,j\leq s}\in\R^{s\times s}$. Note that the semidefinite matrix variables have fixed size $s(\ell,k)$, independent of $d\in\N$. This is a crucial feature for computational efficiency of the approach. 

The dual of the semidefinite program (\ref{primal}) reads:

\begin{equation}\label{dual-aux_{1}}
\begin{split}
\tilde{q}_d^k := \inf_{\substack{\btheta^1,\ldots,\btheta^p,\y}} \left\{\, L_{\y}(f) \right.\quad \text{s.t. }
& y_\gamma\,=\,\theta^\ell_\gamma,\quad \forall\ell: \gamma\in \mathcal{I}_\ell^{\dmax};\quad \forall\gamma\in\Gamma^{\dmax},\quad y_0\,=\,1,\\
&\M_k(\btheta^\ell)\,\succeq\,0,\quad L_{\btheta^\ell}(h_{\alpha\beta})\,\geq\,0,\quad (\alpha,\beta)\in N^\ell_d;\quad\ell=1,\ldots,p\\
&\left. \btheta^\ell\in\R^{s(\ell,\dmax)},\; \ell=1,\ldots, p,\quad  \y\in\R^{|\Gamma^d|} \right\}.
\end{split}
\end{equation}

By standard weak duality of convex optimization, $\tilde{q}^k_d\geq q^k_d$ for all $d\in\N$. 
Moreover \eqref{dual-aux_{1}} is a relaxation of \eqref{def-pb}. Indeed, if $\hat{\x}$ is feasible in \eqref{def-pb}, define $\theta_\gamma^\ell:=\hat{\x}^\gamma$ for all $\gamma \in \mathcal{I}^{\dmax}_\ell$ and all $\ell\in\{1,\ldots,p\}$. Define $\y$ according to the constraints in \eqref{dual-aux_{1}}. Then $y_0=\hat{\x}^0=1$. Let $v_k^\ell$ be the vector of the canonical (monomial) basis of the vector space $\R[\x;I_\ell]_k$, understood as a function $\R^n\to\R^{s(\ell,k)}$. Then $\M_k(\btheta^\ell) = v_k^\ell(\hat{\x})(v_k^\ell(\hat{\x}))^T \succeq 0$. Finally, regarding the Riesz functionals $L_{\btheta^\ell}(h_{\alpha\beta})=h_{\alpha\beta}(\hat{x})\geq 0$ for all $(\alpha,\beta)\in N^\ell_d$, $\ell=1,\ldots,p$ and $L_{\y}(f) = f(\hat{\x})$. Thus to every feasible point $\hat{\x}$ of $(P)$ corresponds a feasible point of \eqref{dual-aux_{1}} giving the same value $f(\hat{\x})$ and therefore
$f^\ast\geq\tilde{q}^k_d\geq q^k_d$ for all $d\in\N$. In fact we have an even more precise and interesting result:

\begin{theorem}[\cite{lasserre-new}]\label{convergence}
Consider problem \eqref{def-pb} and let Assumption \ref{assumption1} \& \ref{assumption2} hold. Let $k\in\N$ be fixed. Then the sequence $(q^k_d)_{d\in\N}$, defined in \eqref{aux_{1}} is monotone non-de\-crea\-sing and $q^k_d\to f^\ast$ as $d\to\infty$.
\end{theorem}
\begin{proof}
Monotonicity of the sequence $(q^k_d)_{d\in\N}$ follows from its definition. Let $\varepsilon>0$ be fixed arbitrary. Then the polynomial $f-f^*+\varepsilon$ is positive on $\K$. By Theorem \ref{sparse-stengle} there exist some polynomials $f^1\ldots,f^p$ such that \eqref{sparse-stengle-1} holds, i.e.,
\[
f-\underbrace{(f^*-\varepsilon)}_{t}\,=\,\sum_{\ell=1}^pf^\ell\quad\mbox{with}\quad f^\ell\,=\,\sum_{(\alpha,\beta)\in N^{\ell}}c_{\alpha\beta}^\ell\,h_{\alpha\beta},\quad\ell=1,\ldots,p,
\]
for finitely many positive weights $c^\ell_{\alpha\beta}$.
Hence $(f^*-\varepsilon,f^\ell,c_{\alpha\beta}^\ell)$ is a feasible solution for \eqref{aux_{1}} as soon as $d$ is sufficiently large, and therefore $q^k_d\geq f^*-\varepsilon$.
Combining this with $q^k_d\leq f^*$ and noting that $\varepsilon>0$ was arbitrary, yield the desired result $q^k_d\to f^*$ as $d\to\infty$.
\end{proof}

Note that we optimize over all possible representations of $f$ of the form $f=\sum_{\ell=1}^p \fl$ with $\fl\in\R[\x;I_\ell]$. By Assumption \ref{assumption1} such a representation exists; however it does not need to be unique.

We next show that a distinguishing feature of the dense BSOS hierarchy \cite{bsos} is also valid for its sparse version Sparse-BSOS.

\begin{theorem}
\label{sosconvexproperty}
Assume the feasible set \K\ in problem \eqref{def-pb} is non-empty and let Assumption \ref{assumption1} \& \ref{assumption2} hold. Let $k\in\N$ be fixed and assume that for every $\ell=1,\ldots,p$, the polynomials $f^\ell$ and $-g_j$ are all SOS-convex polynomials of degree at most $2k$.
(If $k>1$ we assume (with no loss of generality) that for each $\ell=1,\ldots,p$, and some sufficiently large $\kappa_\ell>0$, 
the redundant (SOS-convex) constraints $1-\kappa_\ell^{-1}\sum_{i\in I_\ell}x_i^{2k}\geq0$, $\ell=1,\ldots,p$,
are present in the description (\ref{setkk}) of $\K$.)

Then the semidefinite program \eqref{dual-aux_{1}} has an optimal solution $((\btheta^{*\ell}),\y^*)$ such that
$f^*=\tilde{q}^k_{1}=L_{\y^*}(f)$ and $\x^*:=(L_{\y^*}(x_1),\ldots,L_{\y^*}(x_n))\in\K$ is an optimal solution of (\ref{def-pb}).
Hence finite convergence takes place at the first step of the hierarchy.
\end{theorem}
\begin{proof}
Let $d=1$ (so that $\dmax=2k$) and consider the semidefinite program \eqref{dual-aux_{1}}. Note, that $\btheta^\ell=(\theta^\ell_\gamma)_{\gamma\in\mathcal{I}^{2k}_\ell}$. Recall that by Assumption \ref{assumption2}, for every $\ell=1,\ldots,p$, there exists $j\in J_\ell$ such that $g_j(\x)=1-M_\ell^{-1}\sum_{i\in I_\ell}x_i^2$. In addition if $k>1$ then there also exists $r$ such that $g_r(\x)=1-\kappa_\ell^{-1}\sum_{i\in I_\ell}x_i^{2k}$. 
Hence, feasibility in \eqref{dual-aux_{1}} (with an appropriate choice of $(\alpha,\beta)\in N^\ell_{d}$) implies that
$L_{\btheta^{\ell}}(g_j)\geq0$ and $L_{\btheta^{\ell}}(g_r)\geq0$, which in turn imply:
\[L_{\btheta}(x_i^2)\leq M_\ell\,\theta_0^\ell \:(=M_\ell)\quad\mbox{and}\quad
L_{\btheta}(x_i^{2k})\leq \kappa_\ell\,\theta_0^\ell \:(=\kappa_\ell)\quad\mbox{(if $k>1$)},\quad\forall i\in I_\ell,\quad \forall\ell=1,\ldots,p,\]
where we have used that $\theta^\ell_0=y_0=1$ for all $\ell=1,\ldots,p$.

Combining this with $\M_k(\btheta^\ell)\succeq0$ and invoking Proposition 2.38 in \cite{lasserre-cambridge} yields that $\vert\theta^\ell_\gamma\vert\leq\max[M_\ell,\kappa_\ell,1]$ for every $\vert\gamma\vert\leq 2k$ and $\ell=1,\ldots,p$. 
Consequently, the set of feasible solutions $(\btheta^\ell,\y)$ of \eqref{dual-aux_{1}} is bounded, hence compact. This implies that \eqref{dual-aux_{1}} has an optimal solution $(\btheta^{*\ell},\y^*)$. Notice that among the constraints $L_{\btheta^{*\ell}}(h_{\alpha\beta})\geq0$ are the constraints $L_{\btheta^{*\ell}}(g_j)\geq0$ for all $j\in J_\ell$. As $f^\ell$ and $-g_j$ are SOS-convex, invoking Lemma \ref{lemma-jensen} yields
\[f^\ell(\x^*_\ell)\,\leq\,L_{\theta^{*\ell}}(f^\ell)\quad\mbox{and}\quad 0\,\leq\,L_{\btheta^{*\ell}}(g_j)\,\leq\,g_j(\x^*_\ell),\quad\forall j\in J_\ell,\quad \ell=1,
\ldots,p,\]
where $\x^*_\ell:=(L_{\btheta^{*\ell}}(x_i))_{i\in I_\ell}\in\K_\ell$, $\ell=1,\ldots,p$. In addition, the constraint $y^*_\gamma\,=\,\theta^{*\ell}_\gamma$, for all $\ell$ such that $\gamma\in \mathcal{I}_\ell^{\dmax}$, implies that $(\x^*_{\ell})_i=(\x^*_{\ell'})_i$ whenever $i\in I_\ell\cap I_{\ell'}$. Therefore defining $x^*_i\,:=\,(\x^*_{\ell})_i$ whenever $i\in I_\ell$, one obtains $g_j(\x^*)\geq0$ for all $j$, i.e., $\x^*\in\K$.
Finally,
\[f^*\geq \tilde{q}^k_{1}\,=\,L_{\y^*}(f)\,=\,\sum_{\ell=1}^pL_{\btheta^{*\ell}}(f^\ell)\,\geq\,\sum_{\ell=1}^pf^\ell(\x^*_\ell)\,=\,f(\x^*),\]
which shows that $\x^*\in\K$ is an optimal solution of \eqref{def-pb}. Hence $f^*=f(\x^*)=\tilde{q}^k_{1}$.
\end{proof}

\subsection{Sufficient condition for finite convergence}\label{optimality}

\label{subsec-sufficient-condition}
By looking at the dual \eqref{dual-aux_{1}} of the semidefinite program \eqref{primal} one obtains a sufficient condition for finite convergence. Choose $\omega\in\N$ minimal such that $2\omega\geq \max\{\deg(f),\deg(g_1),\ldots,\deg(g_m)\}$. We have the following lemma.

\begin{lemma}
\label{lemma-dual}
Let $(\btheta^{*1},\ldots,\btheta^{*p},\y^*)\in\R^{s(1,\dmax)}\times\cdots\times\R^{s(p,\dmax)}\times\R^{s}$ 
be an optimal solution of \eqref{dual-aux_{1}}. 
If ${\rm rank}\,\M_\omega(\btheta^{*\ell})=1$ for every $\ell=1,\ldots,p$, then $\tilde{q}^k_d=f^*$ and $\x^*=(y^*_{\boldsymbol{\gamma}})_{\vert\gamma\vert=1}$ is an optimal solution of problem $(P)$.
\end{lemma}
\begin{proof}
If ${\rm rank}\,\M_\omega(\btheta^{*\ell})=1$, then $(\theta^{*\ell}_\gamma)_{\vert\gamma\vert\leq 2\omega}$, is the vector of moments (up to order $2\omega$) of the Dirac measure $\delta_{\x^\ell}$ at the point 
$\x^\ell:=(\theta^{*\ell}_\gamma)_{\vert\gamma\vert=1}\in\R^{n_\ell}$. 
Note that from the constraints $y_\gamma\,=\,\theta^{\ell}_\gamma$ in \eqref{dual-aux_{1}}
\[
x^{\ell_1}_\gamma = x^{\ell_2}_\gamma,\quad \forall \ell_1,\ell_2: \gamma \in \mathcal{I}_{\ell_1}^1\cap \mathcal{I}_{\ell_2}^1.
\]
Hence we can define $x^*_\gamma := x^\ell_\gamma$ for all $\gamma$ such that $\vert\gamma\vert=1$ and independent of the specific choice of $\ell$ such that $\gamma \in \mathcal{I}_{\ell}$. For the same reason $\x^*=(y^*_{\boldsymbol{\gamma}})_{\vert\gamma\vert=1}$.
Consequently, for all $q\in\mathbb{R}[\x;I_\ell]_{2\omega}$
\[
q(\x^*) = q(\x^\ell) = \int q\;\delta_{\x^\ell} = L_{\btheta^{*\ell}}(q)\,=\,L_{\y^*}(q) 
\]
Let $j\in J_\ell$. The constraints $L_{\btheta^{*\ell}}(h_{\alpha\beta})\geq0$ imply in particular $L_{\btheta^{*\ell}}(g_j)\geq0$. Since ${\rm deg}(g_j)\leq 2\omega$, $0\leq L_{\btheta^{*\ell}}(g_j)=g_j(\x^\ell)=g_j(\x^*)$, and so as $j\in J_\ell$ was arbitrary, $\x^*\in\K$.
Finally, and again because ${\rm deg}(f)\leq 2\omega$,
\[f^*\,\geq\,\tilde{q}^k_d\,=\,L_{\y^*}(f)\,=\,f(\x^*),\]
from which we may conclude  that $\x^*\in\K$ is an optimal solution of problem \eqref{def-pb}.
\end{proof}

\section{Computational issues}
\subsection{Comparing coefficients}
As already outlined earlier we have implemented polynomial equality constraints in \eqref{aux_{1}} by comparison of coefficients. The resulting constraints in the SDP are sparse and can be treated efficiently by the SDP solver. A crucial issue for the implementation of the Sparse-BSOS relaxations is how to equate the coefficients. The bottleneck for such an implementation is that one has to gather all occurrences of the same monomials.

As in \cite{bsos}, we use the following data format for representing a polynomial $f$ in $n$ variables: 
\[
\verb|F(i,1:n+1)| = [\gamma^T,\cf_\gamma],
\]
stating that $\cf_\gamma$ is the $i$th coefficient of $f$ corresponding to the monomial $\x^\gamma$. Adding two polynomials is done by concatenating their representations. Hence, equating the coefficients of $\x^\gamma$ is basically finding all indices $i$ of a polynomial $F$, such that $\verb|F(i,1:n)| = \gamma^T$. 

Matlab is providing the function \verb|ismember(A,B,'rows')| to find a row \verb|A| in a Matrix \verb|B|. This however is too slow for our purpose. Instead of using this function, we reduce the problem to finding all equal entries of a vector, which can be handled much more efficiently. To that end we multiply \verb|F(:,1:n)| by a random vector. Generically this results in a vector whose entries are different if and only if the corresponding rows in \verb|F(:,1:n)| are different. 

\subsection{Reducing problem size} \label{sec:removevariables}
By looking at (\ref{primal}) more closely one may reduce the number of free variables and the number of constraints. 
It is likely that there are some indices $i\in \{1,\ldots,n\}$, that only appear in one of the $I_\ell$, say $i\in I_{\ell_i}$. Hence, for all $\gamma\in \bigcup_{j=1}^p \mathcal{I}^{\dmax}_j$ such that $\gamma_i\neq 0$ the second equality constraint in (\ref{primal}) reduces to $\cf^{\ell_i}_\gamma = f_\gamma$. 
Consequently, there is a number of variables that are or can be fixed from the beginning. We do this in our implementation. However, in order to be able to certify optimality by Lemma \ref{lemma-dual}, one needs to trace back these substitutions, to recover the moment sequences $\y^\ell$ from the solution of the dual problem. Removing these fixed variables occasionally leads to equality constraints $0=0$ in the SDP. We remove those constraints for better conditioning.

\section{Numerical experiments}
\label{num-examples}
In this section we provide some examples to illustrate the performance of our approach and to compare  with others. It is natural to compare Sparse-BSOS with its dense counterpart BSOS \cite{bsos} whenever possible. We also compare Sparse-BSOS with the sparse standard SOS-hierarchy \cite{waki} (Sparse-PUT). We use the following notation to refer to both the different SDP relaxations of Problem \eqref{def-pb} and to the particular implementations used in our experiments. Note that for the sparse versions, we suppose that Assumption \ref{assumption1} holds.
\begin{itemize}
\item BSOS: The implementation of the dense version of BSOS as described in \cite{bsos}. 
\[
\begin{split} 
\sup_{t,Q,\blambda}\left\{\,t\quad\right.\text{s.t. } 
& f(x^{(\tau)})  -\sum_{(\alpha,\beta)\in \N^{2m}_d} \lambda_{\alpha\beta}\,h_{\alpha\beta}(x^{(\tau)}) - \inprod{Q}{\left(v_k(x^{(\tau)}) (v_k(x^{(\tau)}))^T\right)} = 0,\\ 
& \left.  \tau=1,\ldots,|\N^n_{\dmax}|, \quad Q \in \cS^{|\N^n_k|}_+,\; \blambda\in\R^{|\N^{2m}_d|}_+,\;t\in\R \right\},
\end{split}
\]
where $d$ is the relaxation order, $k$ is a parameter fixed in advance to control the size of the psd variable $Q$, $v_k$ is the vector of the canonical (monomial) basis of the vector space $\R[\x]_k$  understood as a function $\R^n\to\R^{|\N^n_k|}$, and $\{ x^{(\tau)}\;:\; \tau=1,\ldots,|\N^n_{\dmax}|\}$ is a set of generic points in $[-1,1]^n$.
We emphasize that BSOS uses sampling to implement the equality constraints in the above problem (see discussion \S\ref{SDP-formulation}). We do the same implementation as is used for the numerical experiments in \cite{bsos}. As explained in \cite[Section 3]{bsos}, some of the constraints in the SDP stated above might by nearly redundant, i.e. they might be \enquote{almost} linearly dependant of others. Such constraints are removed before handing the problem over to the SDP solver. A close look to the code also reveals that the maximum number of point evaluations considered is limited to approximately $5000$. This choice has been made to prevent the solver from running out of memory. As a consequence, when $|\N^{2m}_d| > 5000$ the implemented SDP is a relaxation of the relaxation stated above.

\item Sparse-BSOS: Our implementation of the sparse version of BSOS as described in \eqref{primal}.
 
Again $d$ is the relaxation order, $k$ is a parameter fixed in advance to control the size of the psd variables $Q^\ell$.
We use the technique described in the previous section to reduce the size of the SDP before handing it over to the SDP solver. However we do not remove any information from the SDP.
\item Sparse-PUT: Our implementation of the sparse version of the standard SOS-hierarchy \cite{waki}.  
\[
\begin{split}
\sup_{\substack{t,\boldsymbol{Q}^1,\ldots,\boldsymbol{Q}^p,\\\cf^1,\ldots,\cf^p}}\left\{\,t\quad\right.\text{s.t. } 
& \sum_{\ell:\:\gamma\in \mathcal{I}_{\ell}^{\dmax}}\cf^\ell_{\gamma} = \cf_{\gamma},\quad\forall \gamma\in\Gamma^{2d}\backslash\{\boldsymbol{0}\},\quad \sum_{\ell=1}^p\cf^\ell_{\boldsymbol{0}} = \cf_{\boldsymbol{0}}-t \\
& \cf^\ell_\gamma - \inprod{Q_0^\ell}{\left(\bv_d^\ell (\bv_d^\ell)^T\right)}_\gamma - \sum_{j\in J_\ell}\inprod{Q_j^\ell}{\left(\bv_{d_j}^\ell (\bv_{d_j}^\ell)^T\right)}_\gamma = 0,\quad\\ 
& \hspace*{5cm}\forall \gamma\in\mathcal{I}_\ell^{2d},\quad \ell=1,\ldots,p, \\ 
&  Q_0^\ell\in\cS^{s(\ell,d)}, Q_j^\ell \in \cS^{s(\ell,d_j)}_+, j\in J_\ell,\;\ell=1,\ldots,p,\\
& \left.\boldsymbol{\cf}^\ell\in \R^{s(\ell,\dmax)},\;\ell=1,\ldots,p,\;t\in\R \right\},
\end{split}
\]
where $d$ is the relaxation order and $d_j$ is the largest integer less or equal to $\frac{2d-\deg(g_j)}{2}$. We use the same technique to reduce the size of the SDP as for Sparse-BSOS. 
\end{itemize}

The aim of the experiments is twofold. On the one hand we compare Sparse-BSOS with BSOS on small and medium size problems (as BSOS cannot handle large scale problems) with different sparsity patterns. In particular we are interested in the following issues:
\begin{itemize}
\item the influence of the block sizes (depending on the sparsity pattern) when the size of overlaps between blocks of variables is fixed.
\item the influence of various block and overlap sizes for a fixed number of variables ($n=90$).
\item does the finite convergence of the dense version occur systematically earlier than for the sparse version?
(As it cannot occur later.)  
\end{itemize}
On the other hand we compare Sparse-BSOS with Sparse-PUT on high degree small size and lower degree medium and large scale problems. This comparison requires some care because the feasible set for Sparse-PUT is $\K=\{\x:g_j(\x)\geq 0 \}$ while for Sparse-BSOS (and BSOS) it $\K=\{\x:0 \leq g_j(\x)\leq 1 \}$. Hence we code the information about the feasible set in the constraints $0 \leq g_j(\x)$ and scale the constraint polynomials $g_j$ to be less than $1$ on the feasible set. 
In general one expects that if Sparse-PUT gives a good result, Sparse-BSOS will not do better. However we have identified at least three scenarii where Sparse-BSOS can beat Sparse-PUT (and does it at least in some examples). 
\begin{itemize}
\item The first possible Sparse-PUT relaxation\footnote{In Sparse-PUT the maximum size of the psd variables is ${n^*+d\choose n^*}\times{n^*+d\choose d}$ where $n^*=\max_\ell\vert I_\ell\vert$ and $2d\geq \max[{\rm deg}(f),{\rm deg}(g_j)]$.}
yields the optimal value of the polynomial optimization problem, and some degrees $d_j$ of the SOS weights are potentially greater than $0$. This happens, when the degree of the objective function is larger than the degree of some constraint  plus $2$. Then setting the parameter $k=\deg(f)/2$, the first Sparse-BSOS relaxations ($d=1,2,\ldots$) are faster than Sparse-PUT and may also reach the optimal value; this is illustrated in Tables \ref{tab:chainwood} \& \ref{tab:chainsing}.
\item The first possible Sparse-PUT relaxation does not reach the optimal value of the POP and the second relaxation cannot be solved (because its size is too large and/or is too costly to implement). If the SOS weights in the first Sparse-PUT relaxation are all of degree $0$, then again setting the parameter $k=\deg(f)/2$, the first Sparse-BSOS relaxation gives the same result and it is possible to obtain better bounds by going higher in the relaxation order. In particular, this is the case for the important class of quadratic/quadratically constrained programs (that is when
$\max[{\rm deg}(f),{\rm deg}(g_j)]\leq 2$); this is illustrated in Table \ref{tab:QP_quad}.
\item The first possible Sparse-PUT relaxation cannot be solved. Then setting the parameter $k<\deg(f)/2$, the first Sparse-BSOS relaxations ($d=1,2,\ldots$) may be solvable and so provide lower bounds on the optimal value of the polynomial optimization problem whereas Sparse-PUT cannot; this is illustrated in Table \ref{tab:QP_quat}.
\end{itemize}
To summarize the above cases, in Sparse-BSOS we take full advantage of the facts that (a) the constraints enter the certificate only with non-negative weights in contrast to SOS weights in Sparse-PUT, (b) the size of the psd variables is fixed and does not increase with the relaxation order. In particular, 
while the minimal size of the largest psd variable in Sparse-PUT is determined by the polynomial data, in Sparse-BSOS one can always set the parameter $k$ to $1$ which implies that the maximum size of the semidefinite matrices in the SDP \eqref{primal} is always at most  $O(n^*)$ where $n^*=\max_\ell n_\ell$, for all $d$. This is because by assumption, the $g_j$ generate the algebra $\R[\x]$ and so a polynomial of arbitrary degree and positive on $\K$, can be obtained as a positive linear combination of the $g_j^{\alpha_j}(1-g_j)^{\beta_j}$ (with no SOS involved). So even if $\max[{\rm deg}(f),\max_j{\rm deg}(g_j)]>2$, the optimal value of \eqref{aux_{1}} (with $k=1$) is finite as soon as $d$ is large enough, and so provides a non trivial lower bound.

The results on numerical experiments described in the next sections are biased by the (limited) sample of examples that we have considered. Therefore they should be understood as partial indications rather than definite conclusions. The latter would require much more computational experiments beyond the scope of the present paper.\\

All experiments were performed on an Intel Core i7-5600U CPU @ 2.60GHz $\times$ 4 with 16GB RAM. Scripts are executed in Matlab 8.5 (R2015b) 64bit on Ubuntu 14.04 LTS operating system. \label{computer} The SDP solver used is SDPT3-4.0\cite{TTT99,TTT03}.

The results are presented in tables below. They provide the following information:
\begin{itemize}
\item A pattern or problem code to identify the example.
\item The relaxation order $d$ and the chosen parameter $k$ for the psd constraints.
\item The maximal degree $\dmax$, appearing in the certificate.
\item The numbers of non-negative variables (corresponding to $\lambda_{\alpha\beta}$), unrestricted (free) variables (corresponding to $t$ and the coefficients of the $\fl$), and the number (and size) of the positive semidefinite variables (corresponding to the sums of squares).
\item The number of (equality) constraints in the SDP.
\item The (primal) solution of the SDP.
\item The time in seconds, including the times to generate and solve the SDP as well as computing the optimality condition.
\item The abbreviation rk stands for the rank of the moment matrices according to Section \ref{optimality}. In the case of Sparse-BSOS and Sparse-PUT, rk is the average rank of all moment matrices and can hence be a decimal number. When reporting an integer the rank is acutally integer, i.e. if we write rk$ = 1$, the rank is actually $1$, if however we write $1.0$ the rank is strictly bigger than $1$.
\item If the primal solution is written in bold, it was certified by the rank condition and coincides\footnote{In this context we consider two numbers to be equal if there difference is less than $10^{-8}$.} with the global optimum of the POP.
\item Primal solutions were marked with * when the solver stopped because \textit{steps were too short}, the \textit{maximum number of iterations} was achieved, or \textit{lack of progress}. In these cases one has to consider the result carefully.
\end{itemize}

\subsection{BSOS vs. Sparse-BSOS}

\subsubsection{Dense small size examples }\label{sec:dense}
\textbf{Table \ref{tab:dense}:} We compare the sparse and the dense version of BSOS on a set of examples introduced in \cite{bsos}. In the problem description the first number of the name indicates the number of variables, the second the degree of the problem. The examples are relatively small size, i.e. $n\leq 20$. The degree of the objective function and the constraints is between $2$ and $8$. As the test sample is from the dense version, no sparsity pattern is present and we pass the information $I=\{1,\ldots,n\}$ and $J = \{1,\ldots,m\}$ to Sparse-BSOS. Consequently both hierarchies compute the same certificate. The only difference comes from the implementation of the equality constraints and the different handling of the psd variable in SDPT3 (see discussion p. \ref{SDP-formulation}).

\begin{table} 
\begin{tabular}{|c|c|c||ccc||ccc|} 
 \hline 
		&       &		&  \multicolumn{3}{c||}{BSOS} & \multicolumn{3}{c|}{Sparse-BSOS}\\
problem	& (d,k) &$\dmax$& solution       			& rk & time & solution  				& rk & time\\ 
 \hline 
 P4\_2 	& (1,1) &	2	& \textbf{-5.7491e-01}  	&  1 &  0.8s	 & \textbf{-5.7491e-01}	&  1 &   1.6s \\ 
\hline 
 P4\_4 	& (1,2) &	4	& -6.5919e-01  			&  7 &  0.3s & -6.5919e-01  			&  3 &   0.4s \\ 
  		& (2,2) &	8	& \textbf{-4.3603e-01}  	&  1 &  0.7s & \textbf{-4.3603e-01} 	&  1 &   0.5s \\ 
\hline 
 P4\_6 	& (1,3) &	6	& -6.2500e-02* 			& 27 &  1.0s	 & -6.2500e-02  			& 15 &   0.5s \\ 
  		& (2,3) &	12	& -6.0937e-02  			&  7 &  0.7s	 & -6.0937e-02* 			&  6 &   0.6s \\ 
		& (3,3) &	18	& -6.0693e-02  			&  4 &  2.6s	 & -6.0693e-02* 			&  4 &   4.7s \\

\hline 
 P4\_8 	& (1,4) &	8	& -9.3381e-02* 			& 39 &  9.2s	 & -9.3355e-02  			& 15 &   1.7s \\ 
  		& (2,4) &	16	& -8.5813e-02*  			&  9 &  3.0s	 & -8.5813e-02  			&  4 &   1.3s \\ 
		& (3,4) &	24	& -8.5813e-02  			&  4 &  4.3s	 & -8.5814e-02* 			&  4 &   4.1s \\
\hline 
 P6\_2 	& (1,1) &	2	& \textbf{-5.7491e-01}  	&  1 &  0.2s	 & \textbf{-5.7491e-01}	&  1 &   0.3s\\ 
\hline 
 P6\_4 	& (1,2) &	4	& -5.7716e-01  			& 13 &  0.7s	 & -5.7716e-01  			& 4 &   0.4s\\ 
  		& (2,2) &	8	& -5.7696e-01  			&  4 &  4.4s	 & -5.7696e-01  			& 3 &   0.7s\\ 
  		& (3,2) &	12	& -5.7696e-01  			&  3 & 25.0s	 & -5.7765e-01* 			& 3 &  16.6s\\ 
\hline 
 P6\_6 	& (1,3) &	6	& -6.5972e-01*  			& 35 &  6.6s	 & -6.5972e-01  			&  7 &   2.7s\\ 
 		& (2,3) &	12	& -6.5972e-01*  			& 32 & 21.5s	 & -6.5972e-01  			&  4 &   4.4s\\ 
 		& (3,3) &	18	& \textbf{-4.1288e-01*}	&  1 & 44.5s & \textbf{-4.1288e-01*}	&  1 &  82.1s \\ 
\hline 
 P8\_2 	& (1,1) &	2	& \textbf{-5.7491e-01}  &  1 &  0.2s	 & \textbf{-5.7491e-01}	&  1 &   0.3s\\ 
\hline 
 P8\_4 	& (1,2) &	4	& -6.5946e-01  			& 21 &  1.5s	 & -6.5946e-01  			& 5 &   0.8s\\ 
  		& (2,2) &	8	& \textbf{-4.3603e-01*}	&  1 & 17.7s	 & \textbf{-4.3603e-01*}	& 1 &   2.7s\\ 
\hline 
 P10\_2 & (1,1) &	2	& \textbf{-5.7491e-01}	&  1 &  0.3s	 & \textbf{-5.7491e-01}	& 1 &   0.3s\\ 
\hline 
 P10\_4 & (1,2) &	4	& -6.5951e-01  			& 31 &  5.5s	 & -6.5951e-01  			& 6 &   2.2s \\ 
        & (2,2) &	8	& \textbf{-4.3603e-01*}	&  1 & 23.4s	 & \textbf{-4.3603e-01*}	& 1 &   8.4s\\ 
\hline 
 P20\_2 & (1,1) &	4	& \textbf{-5.7492e-01*} 	&  1 & 0.8s	 & \textbf{-5.7491e-01} 	& 1 &   0.4s\\ 
\hline 
\end{tabular}
\caption{Comparison BSOS vs. Sparse-BSOS on non sparse examples}
\label{tab:dense}
\end{table}

We see that Sparse-BSOS is able to solve the same problems as BSOS. As the problems are dense, Sparse-BSOS uses a trivial sparsity pattern and both 
certificates are the same. Consequently the optimal value coincides unless one of the SDPs stopped because of numerical issues. In most cases Sparse-BSOS is faster than BSOS.

\subsubsection{Sparse quadratic examples (medium and large scale)}
For the remaining examples in \S \ref{num-examples} we consider sparsity patterns having some banded structure. The patterns are described by a vector $\n\in\N^p$ and a natural number $\oo$. The vector $\n$ determines the size of the blocks $I_\ell$ whereas  $\oo$ defines the number of overlapping variables between two consecutive blocks. More formally defining $c_1:=n_1$ and $c_\ell := c_{\ell-1}+n_{\ell}-\oo$ we construct
\[
I_\ell := \{c_{\ell}-n_\ell+1,\ldots,c_\ell\}.
\]
Note that the total number of variables in pattern $I$ is $c_p$. We call those sparsity pattern banded, because the RIP (\see Assumption \ref{assumption1}) is satisfied by  
\[
\left(I_{\ell+1}\cap\bigcup_{r=1}^\ell I_r\right) \subseteq I_\ell.
\]
Informally for $\n$ we use notations like $(7\times 5)$ instead of $(5,5,5,5,5,5,5)$ or $(2\times 17,13)$ instead of $(17,17,13)$ without any misunderstanding. 

We also analyze the impact of different sparsity pattern on instances of the following sparse quadratic optimization problem. Given a sparsity pattern $I=\{I_1,\ldots,I_p\}$ we consider
\begin{equation}\label{eq:QP}
\min_\x\left\lbrace x^TAx + b^Tx\;:\;  1 - \sum_{i\in I_\ell}x_i^s\geq 0 \quad \ell = 1,\ldots,p,\quad x_i \geq 0 \quad i = 1,\ldots,n,\right\rbrace,
\tag{QP}
\end{equation}
where $b$ is a random vector and the symmetric matrix $A$ is randomly generated according to $I$\footnote{By this we means that a random value between $-1$ and $1$ is assigned to an entry $a_{ij}$ of $A$ if and only if both $i$ and $j$ are contained in the same $I_\ell$ for some $\ell$. Otherwise $a_{ij}=0$. The values of $b$ are randomly generated between $-1$ and $1$, too.}. We verify that $A$ has positive and negative eigenvalues to make sure, that our problem is non-convex. Depending on the choice of $s\in\{1,2\}$ we call the constraints \textit{linear} or \textit{quadratic}, although in the latter case we still have the linear constraints $x_i\geq 0$. Note that the second constraint implies that the first constraint is at most $1$ and vice versa.\footnote{Also note that Assumption 2 is fulfilled only when $s=2$. In the case of $s=1$ with the same arguments as in the first part of the proof of Theorem \ref{sosconvexproperty}, one can show that in this case with $d \geq 2$, the feasible set of the SDP \eqref{dual-aux_{1}} is compact and an optimal solution is attained. Furthermore, since the set $\K$ is a polyhedron, the quadratic modules associated to the $K_\ell$ are archimedean. One can adapt the proofs of Theorem \ref{sparse-stengle} and \ref{convergence} so that the convergence result is still true in the case $s=1$.} \label{sec:footnote}

\vspace{11pt}
\textbf{Table \ref{tab:qp1}:} To compare the respective SDPs arising from BSOS and Sparse-BSOS we fix $\n=(2\times 50)$ and create different sparsity patterns by varying $\oo$. From these patterns we generate instances of \eqref{eq:QP} with $s=2$. Choosing parameter $k=1$ for the size of the sum of squares we compute and solve the first relaxation $d=1$ of BSOS and Sparse-BSOS. As $\n$ is fixed for all examples the number of variables $n$ grows when the overlap $\oo$ decreases.

\begin{table}
\begin{tabular}{|l|c|r|c|c|r|r|r|}
\hline
$\oo/n$	&		& \# n-neg. 	&\# free  	& \# psd var.	& \# constr. & solution	& time \\
		&		&	var.		&	var.		& (size) 		&			 &			&		\\
\hline	
40/60 	& BSOS	& 125		& 1			& 1(61)			& 1891		&-1.1123e+01		& 13.8s\\
		& Sparse-BSOS	& 206		& 1723		& 2(51)			& 3513		&-1.1123e+01		& 14.0s\\
\hline
30/70 	& BSOS	& 145		& 1			& 1(71)			& 2556		&-1.2753e+01		& 24.3s\\
		& Sparse-BSOS	& 206		& 993		& 2(51)			& 3148		&-1.2753e+01		& 11.2s\\
\hline
20/80 	& BSOS	& 165		& 1			& 1(81)			& 3321 		&-1.3376e+01		& 48.1s\\
		& Sparse-BSOS	& 206		& 463		& 2(51)			& 2883		&-1.3376e+01		& 10.5s \\
\hline
10/90 	& BSOS	& 185		& 1			& 1(91)			& 4186		&-1.5406e+01		& 73.6s \\
		& Sparse-BSOS	& 206		& 133		& 2(51)			& 2718		&-1.5406e+01		&  9.3s \\
\hline
5/95 	& BSOS	& 195		& 1			& 1(96)			& 4656		&-1.5665e+01		& 89.5s\\
		& Sparse-BSOS	& 206		& 43			& 2(51)			& 2673		&-1.5665e+01		&  9.2s\\
\hline
1/99 	& BSOS	& 203		& 1			& 1(100)			& 5050		&-1.5658e+01	*	&152.2s \\
		& Sparse-BSOS	& 206		& 7			& 2(51)			& 2655		&-1.5658e+01		& 10.8s\\
\hline
\end{tabular}
\caption{QPI $\n=(50,50)$, quadratic constraints: $s=2$, maximal degree of the certificate $\dmax=2$, time to compute the first relaxation $d=1$ with $k=1$}
\label{tab:qp1}
\end{table}

Both BSOS and Sparse-BSOS are able to solve all instances of this problem and provide the same lower bounds. In contrast to the previous example the certificates and the corresponding SDP handed over to the solver are different:
Consider the example with $\oo=40$ and $n=60$ variables. As $k=1$ the psd variable in BSOS is of size $\binom{n+k}{k} = 61$ and grows with the number of variables. As the sparsity pattern in all examples consists of two blocks of $50$ variables, Sparse-BSOS always has $2$ psd variables of size $\binom{n_\ell+k}{k} = 51$, independently of the total number of variables. 
The unrestricted variable in BSOS corresponds to the optimizing variable $t$. Sparse-BSOS also has this optimizing variable. The other unrestricted variables correspond to the non-fixed coefficients of the polynomials $\fl$. This is easy to see in the case of $\oo=1$: The maximal degree of the certificate is $\dmax=2$. Hence, the non-fixed coefficients are the coefficients of the monomials $1, x_{50}$ and $x_{50}^2$. Consequently after removing the unrestricted variables for the fix coefficients, we have $3$ unrestricted variables for $f^1$ and $3$ for $f^2$. Together with the optimizing variable, we end up with $7$ unrestricted variables as presented in the table.
The non-negative variables in Table \ref{tab:qp1} correspond to the $\lambda_{\alpha\beta}$ in the description of BSOS and Sparse-BSOS. Note that the number of constraints for each block is $m_1=m_2=51$ for Sparse-BSOS and $m=n+2$ for BSOS (the linear constraints plus the two quadratic constraints). Consequently there are $206=\binom{2m_1+d}{d}+\binom{2m_2+d}{d}$ non-negative variables for Sparse-BSOS and $\binom{2(n+2)+d}{d}$ non-negative variables for BSOS depending on the number of variables $n$. 
The number of constraints in BSOS corresponds to the number of point evaluations needed to guarantee the equality constraint in the BSOS formulation, i.e $|\N^n_{\dmax}| = \binom{n+\dmax}{\dmax}$, and hence increases with $n$. For Sparse-BSOS three equalities have to be considered. As they are implemented by comparing coefficients we expect $|\mathcal{I}_1^{\dmax}|+|\mathcal{I}_2^{\dmax}|+|\Gamma_{\dmax}| = 2\times\binom{50+2}{2}+(2\times\binom{50+2}{2}-\binom{\oo+2}{2})$ many constraints. The difference comes from the fact that with every removed unrestricted variable, we also remove a constraint. 

Summarizing, when the overlap $\oo$ decreases Sparse-BSOS benefits from having less unrestricted variables and constraints while the size of the psd variables remains the same; the SDP becomes easier. In contrast to this in BSOS the size of the psd variable and the number of non-negative variables and constraints increases; the SDP becomes harder. The solving time reported in the last column of the table reflects this nearly perfectly.

\vspace{11pt}
\textbf{Table \ref{tab:qp2}:} We create the sparsity pattern $I$ with $\n=(11\times10)$ and $\oo=2$ and consider one instance of Problem \eqref{eq:QP} with $s=1$. According to the footnote in Section \ref{sec:footnote} we can guarantee the existence of a dual solution of Sparse-BSOS by choosing $d\geq 2$. Again, we choose parameter $k=1$ and compute the second BSOS and Sparse-BSOS relaxation $d=2$. For Sparse-BSOS we use different sparsity patterns and observe the effect on the SDP and the solving time.

\begin{table}
\begin{tabular}{|l|c|r|r|c|r|r|}
\hline
		&					& \# n-neg. 	&\# free  	& \# psd var.		& \# constr.	& time 	 \\
		&		$\n$			&	var.		&	var.		& (size) 			&			&		 \\
\hline	
BSOS & none					& 20706		& 1			& 1(91)				& 4186		&   882.4s\\
\hline
Sparse-BSOS& $(90)$			& 20706		& 2			& 1(91)				& 4187		&    49.0s\\
	 & $(50,42)$				& 11001		& 13			& 1(51)/1(43)		& 2278 		&    10.5s\\
	 & $(50,26,18)$			&  9072		& 24			& 1(51)/1(27)/1(19)	& 1905 		&     7.8s\\
	 & $(50,2\times18,10)$	&  8439		& 35			& 1(51)/2(19)/1(11)	& 1788 		&     6.7s\\	 	 
	 & $(50,5\times10)$		&  7821		& 57			& 1(51)/5(11)		& 1682 		&     6.6s\\
	 & $(2\times34,26)$		&  7776		& 24			& 2(35)/1(27)		& 1649 		&     5.3s\\
	 & $(3\times26,18)$		&  6171		& 35			& 3(27)/1(19)		& 1340 		&     3.2s\\	 
	 & $(34,3\times18,10)$	&  5862		& 46			& 1(35)/3(19)/1(11)	& 1287 		&     3.4s\\
	 &$(2\times26,2\times18,10)$	&  5538	& 46			& 2(27)/2(19)/1(11)	& 1223 		&     2.8s\\
	 & $(5\times18,10)$		&  4581		& 57			& 5(19)/1(11)		& 1042 		&     1.7s\\
	 & $(11\times10)$		&  3036		& 112		& 11(11)				&  777		&     0.8s\\
	 \hline
\end{tabular}
\caption{QPII $n=90$, overlap $2$, linear constraints: $s=1$, $(d,k)=(2,1)$, maximal degree of the certificate $\dmax=2$, same optimal solution verified by rank one condition in all cases}
\label{tab:qp2}

\end{table}

The dense and the sparse hierarchy are able to solve this sparse problem and certify optimality by the rank one condition. We do not repeat the discussion on the number of variables and constraints in this second example. We only remark that the additional unrestricted variable and constraints in in Sparse-BSOS for $\n=(90)$ compared to the BSOS case without sparsity pattern comes from the equality $\cf^1_{\boldsymbol{0}} = \cf_{\boldsymbol{0}}-t$. Because $t$ is variable $\cf^1_{\boldsymbol{0}}$ is not fixed and hence cannot be removed like all the other coefficients in this case (cf.\S\ref{sec:removevariables}).

The reason for the big difference of computing times between BSOS and Sparse-BSOS is double. On the one hand side, searching for linear dependent constraints in BSOS takes a lot of time. Generating the SDP with BSOS took over $70$ seconds whereas the SDP for Sparse-BSOS was generated in less than $5$ seconds. The main reason however is hidden in the constraints. Indeed the constraints in Sparse-BSOS are sparse whereas the constraints in BSOS are dense and therefore the SDP solver is much slower in the dense case.\footnote{Regarding this example the implementation of equalities using sampling might look questionable. Its positive effect comes into play when considering larger psd variables. In \cite{bsos} the authors were able to handle problems with psd variables of size up to $861$ ($n=40$, $k=2$) due to the special handling of the constraints associated to the psd variable in the case of sampling. This is by far out of the range of what can be done by comparing coefficients.} 
With regard to the computing times for Sparse-BSOS one can see that the size of the biggest psd variable is a more important factor than the number of non-negative and free variables or the number of constraints.

\vspace{11pt}
\textbf{Table \ref{tab:qp-ls}:} We next use the quadratic problem \eqref{eq:QP} a third time to investigate the range of Sparse-BSOS on large scale examples. In this sample we generate sparsity patterns with $400$ to $1000$ blocks of size $3$ to $9$ and small overlap between $1$ and $3$. As in the previous example we chose $s=1$ and $d=2$. The size of those examples is by far too large to run BSOS and so we only display the results obtained by Sparse-BSOS. They show that Sparse-BSOS is able to compute lower bounds for sparse large scale problems in reasonable time.

\begin{table}
\begin{tabular}{|r|c|c|r|r|c|r|r|r|}
\hline
$\n$		& $\oo$	& $n$ &\# n-neg.var.	&\# unrest.var. 	& \# psd var.(size) 	& \# const 	& rnk 	& time \\
\hline	
 100x4 	& 1		&  301	&   6 600	&     497		&  100( 5)			&  1 699		& 1	 	&   1.8s\\
 400x4	& 1		& 1201	&  26 400	&   1 997		&  400( 5)			&  6 799		& 1.03 	&  11.8s\\
 700x4	& 1		& 2102	&  46 200	&   3 497		&  700( 5)			& 11 899		& 1.02 	&  28.2s\\
1000x4	& 1		& 3001	&  66 000	&   4 997		& 1000( 5)			& 16 999		& 1.03 	&  49.1s\\
 \hline
 100x5 	& 2		&  302	&   9 100	&   1 091		&  100( 6)			&  2 596		& 1.06 	&   2.4s\\
 400x5	& 2		& 1202	&  36 400	&   4 391		&  400( 6)			& 10 396		& 1.04 	&  16.2s\\
 700x5	& 2		& 2102	&  63 700	&   7 691		&  700( 6)			& 18 196		& 1.10 	&  35.1s\\
1000x5	& 2		& 3002	&  91 000	&  10 991		& 1000( 6)			& 25 996		& 1.08 	&  74.8s\\
 \hline
  50x8 	& 2		&  302	&   9 500	&     541		&   50( 9)			&  2 496		& 1	 	&   4.3s\\
 200x8	& 2		& 1202	&  38 000	&   2 191		&  200( 9)			&  9 996		& 1.08 	&  14.9s\\
 350x8	& 2		& 2102	&  66 500	&   3 841		&  350( 9)			& 17 496 	& 1.01 	&  35.3s\\
 500x8	& 2		& 3002	&  95 000	&   5 491		&  500( 9)			& 24 996 	& 1.02 	&  68.7s\\
 \hline
  50x9 	& 3		&  303	&  11 550	&     933		&   50(10)			&  3 192		& 1.08 	&   3.2s\\
 200x9	& 3		& 1203	&  46 200	&   3 783		&  200(10)			& 12 792		& 1.07 	&  18.2s\\
 350x9	& 3		& 2103	&  80 850	&   6 633		&  350(10)			& 22 392		& 1.06 	&  44.6s\\
 500x9	& 3		& 3003	& 115 500	&   9 483		&  500(10)			& 31 992		& 1.04 	& 133.4s\\
 \hline
\end{tabular}
\caption{QPLS, linear constraints: $s=1$, $(d,k)=(2,1)$, maximal degree of the certificate $\dmax=2$}
\label{tab:qp-ls}
\end{table}

Summarizing the computational results of this section, we saw that Sparse-BSOS is competitive with the dense version on dense examples. On sparse examples the Sparse-BSOS outperforms BSOS as it can use the additional information. The advantage becomes bigger, when the block size $n_\ell$ is small with respect to the total number of variables $n$. In addition, the number of variables in the intersection of at least two blocks $I_\ell$ influences the performance of Sparse-BSOS.
Although the certificates depend on the information, known about the sparsity pattern, we did not encounter that the value computed by BSOS or by Sparse-BSOS with a coarse sparsity pattern, was better than the one computed with the actual pattern.

\subsection{Sparse-BSOS vs. Sparse-PUT} \label{sec:Sparse-PUT}
As already mentioned, the sparse version Sparse-PUT \cite{waki} of the standard hierarchy of SOS relaxations  has been proved to be efficient in solving several large scale problems; see for instance its successful application to some Optimal Power Flow problems \cite{molzahn,josz}. However there are a number of cases where only the first relaxation of Sparse-PUT can be implemented because the second one is too costly to implement. Also if $t:={\rm deg}(f)>2$ then the first SDP relaxation is already very expensive (or cannot be implemented) because some moment matrices of size ${n^*+t\choose t}\times{n^*+t\choose t}$ are constrained to be psd (where $n^*:=\max_\ell n_\ell$).

\subsubsection{Test problems from the literature}
In this section we present experiments on some test problems considered to be challenging in non-linear optimization. 
All test functions are sums of squares and share the global minimum $0$. Hence, it would be possible to compute the minimum in the unconstrained case. However, if not using constraints both Sparse-BSOS and Sparse-PUT reduce to searching for sums of squares. Hence, we restrict the problems to the set 
\[
\K = \{ x\in \R^n\;:\; 1-\sum_{i\in I_\ell}x_i\geq 0,\quad \ell=1,\ldots,p; \quad x_i\geq 0,\quad i=1,\ldots,n\},
\]
which depends on the sparsity pattern of the specific function. Consequently, the optimal values of the considered functions are strictly greater than zero, when the minimizer of the unconstrained problem is not in $\K$. We consider the following test functions of degree $4$:
\begin{itemize}
\item The \textit{Chained Wood Function}:
\[
\begin{split}
f:= \sum_{j\in H} \left(100(x_{j+1}-x_{j}^2)^2 + (1-x_{j})^2 + 90(x_{j+3}-x_{j+2}^2)^2 \right.\\
\left.+ (1-x_{j+2})^2 + 10(x_{j+1}+x_{j+3}-2)^2 + 0.1(x_{j+1}-x_{j+3})^2\right)
\end{split}
\]
where $H := \{2i-1 \, :\, i = 1,\ldots,n/2-1\}$ and $n\equiv 0 \mod 4$. The sparsity pattern is given by $\n=(p\times4)$ and $\oo = 2$.
\item The \textit{Chained Singular Function}:
\[
f:= \sum_{j\in H}
\left((x_j+10x_{j+1})^2 + 5(x_{j+2}-x_{j+3})^2  + (x_{j+1}-2x_{j+2})^4 + 10(x_{j}-x_{j+3})^4\right)
\]
where $H := \{2i-1 \, :\, i = 1,\ldots,n/2-1\}$ and $n\equiv 0 \mod 4$. The sparsity pattern is given by $\n=$ $(p\times4)$ and $\oo = 2$.
\item The \textit{Generalized Rosenbrock Function}:
\[
f:= \sum_{i = 2}^n\left(100(x_i-x_{i-1}^2)^2+(1-x_i)^2\right).
\]
The sparsity pattern is given by $\n= (p\times 2)$ and $\oo = 1$.
\end{itemize}

\vspace{11pt}
\textbf{Table \ref{tab:chainwood} \& \ref{tab:chainsing}:} We solve the Chained Wood and the Chained Singular Function for $n=500,\ldots,1000$ with Sparse-BSOS and Sparse-PUT. For Sparse-BSOS we fix $k=2$ and compute the first and the second relaxation. For Sparse-PUT the relaxation $d=1$ is infeasible, as the degree of the certificate is at most $2$ but the functions are of degree $4$. Hence the first feasible relaxation for Sparse-PUT is $d=2$. When reading the tables note that the reported rank is the average of the rank of several matrices and hence is not necessarily an integer.

\begin{table} 
\begin{tabular}{|l|c||crr||crr||crr|} 
\hline 
				&     	& \multicolumn{3}{c||}{Sparse-BSOS}& \multicolumn{3}{c||}{Sparse-PUT}\\
ChainedWood	&  rel. & solution     	& rk & time 	& solution     & rk & time\\	
\hline 
$n=500$ 	& $d=1$ &\textbf{3.8394e+03}*	&1	& 16.7s	&	inf				& - & - \\
		& $d=2$ &\textbf{3.8394e+03}	 	&1  & 10.4s	&\textbf{3.8394e+03}	& 1 & 16.7s	\\
\hline 
$n=600$ 	& $d=1$ &\textbf{4.6104e+03}*	&1& 20.6s	&	inf				& - & - \\
		& $d=2$ &\textbf{4.6104e+03}		&1  & 13.2s	&\textbf{4.6104e+03}	& 1 & 21.0s	\\
\hline 	
$n=700$ 	& $d=1$ &\textbf{5.3813e+03}*	&1  & 24.1s	&	inf				& - & - \\
		& $d=2$ &\textbf{5.3813e+03}		&1  & 15.8s	&\textbf{5.3813e+03}	& 1 & 26.0s	\\
\hline 	
$n=800$ 	& $d=1$ &\textbf{6.1523e+03}*	&1	& 27.3s	&	inf				& - & - \\
		& $d=2$ &\textbf{6.1523e+03}		&1  & 19.4s	&\textbf{6.1523e+03}	& 1 & 31.1s	\\
\hline 	
$n=900$ 	& $d=1$ &\textbf{6.9232e+03}*	&1	& 30.8s	&	inf				& - & - \\
		& $d=2$ &\textbf{6.9232e+03}		&1  & 22.3s	&\textbf{6.9232e+03}	& 1 	& 36.5s	\\
\hline 		
$n=1000$	& $d=1$ &	     7.6942e+03*		&3	& 28.6s	&	inf				& - 	& - \\
		& $d=2$ &\textbf{7.6942e+03}		&1  & 26.1s	&\textbf{7.6942e+03}	& 1 	& 42.3s	\\
\hline 				
\end{tabular}
\caption{Comparison Sparse-BSOS ($k=2$) and Sparse-PUT on the Chained Wood Function}
\label{tab:chainwood}
\end{table}

\begin{table} 
\begin{tabular}{|l|c||crr||crr||crr|} 
\hline 
				&     	& \multicolumn{3}{c||}{Sparse-BSOS}& \multicolumn{3}{c||}{Sparse-PUT}\\
ChainedSingular	&  rel. & solution     & rk & time 	& solution     & rk & time\\	
\hline 
$n=500$ 	& $d=1$ &-1.4485e-02* 			&1.0& 19.6s	&	inf		&	- & - \\
		& $d=2$ &\textbf{-9.7833e-10}  	&1  & 17.8s	&\textbf{-2.0271e-10 } & 1 & 22.6s	\\
\hline 
$n=600$ 	& $d=1$ &-2.7372e-03* 			&1.0& 40.1s	&	inf		&	- & - \\
		& $d=2$ &\textbf{-1.2640e-09}  	&1  & 21.4s	&\textbf{-1.9613e-10}	& 1 & 27.8s	\\
\hline 	
$n=700$ 	& $d=1$ &-1.7548e-03* 			&1.0& 41.6s	&	inf		&	- & - \\
		& $d=2$ &\textbf{-1.7613e-09}  	&1  & 25.3s	& \textbf{-2.4628e-10}	& 1 & 34.1s	\\
\hline 	
$n=800$ 	& $d=1$ &-1.9438e-03* 			&1.0& 58.9s	&	inf		&	- & - \\
		& $d=2$ &\textbf{2.1935e-09}  	&1  & 29.0s	&\textbf{-2.3398e-10}	& 1 & 41.0s	\\
\hline 	
$n=900$ 	& $d=1$ &-1.8924e-02* 			&1.0& 43.5s	&	inf		&	- & - \\
		& $d=2$ &\textbf{-2.6072e-09}  	&1  & 33.5s	& \textbf{-3.5871e-10}	& 1 & 47.3s	\\
\hline 		
$n=1000$	& $d=1$ &-4.4914e-02* 			&1.0& 35.5s	&	inf		&	- & - \\
		& $d=2$ &\textbf{-9.3508e-10}  	&1  & 39.5s	& \textbf{-1.7329e-10}	& 1 & 54.9s	\\
\hline 				
\end{tabular}
\caption{Comparison Sparse-BSOS ($k=2$) and Sparse-PUT on the Chained Singular Function}
\label{tab:chainsing}
\end{table}

For $d=2$ both Sparse-BSOS and Sparse-PUT are able to find and certify the optimal value (up to numerical errors) for both functions. In these examples Sparse-PUT is slower than Sparse-BSOS because for every constraint Sparse-PUT introduces a psd variable of size $5\times 5$ corresponding to a sum of square of degree $2$. Sparse-BSOS only introduces a non-negative variable for all products and squares of constraints. As the number of non-negative variables is not too big, Sparse-BSOS beats Sparse-PUT.

At this point we noticed a rather strange phenomenon. For the first Sparse-BSOS relaxation $d=1$, the SDP solver runs into numerical problems. Yet, from the definition of the Chained Functions we know that they are sum of squares of degree $4$, which in principle Sparse-BSOS is able to represent with $k=2$ for any $d$. One possible explanation is that the solver may be confused by the additional non-negative variables. However, we could reproduce the same behaviour when omitting the constraints and explicitly only searching for a sum of squares. This phenomena is not specific to our implementation or to SDPT3. It also occurs when searching for a sum of squares representation of the Chained Wood Function ($n=4$) with Gloptipoly3 \cite{gloptipoly} using SeDuMi1.3\cite{sedumi} and with Yalmip\cite{yalmip} using Mosek7\cite{mosek}.

\vspace{11pt}
\textbf{Table \ref{tab:rosenbrock}:} Solving the Generalized Rosenbrock Function for $n=100,\ldots,600$ with Sparse-BSOS and Sparse-PUT. As in the previous examples for Sparse-BSOS we fix $k=2$ and compute the first and the second relaxation. Again for Sparse-PUT the relaxation $d=1$ is infeasible. Hence we start the Sparse-PUT relaxation with $d=2$.

\begin{table} 
\begin{tabular}{|l|c||crr||crr||crr|} 
\hline 
				&     	& \multicolumn{3}{c||}{Sparse-BSOS}& \multicolumn{3}{c||}{Sparse-PUT}\\
GeneralizedRosenbrock	&  rel. & solution     & rk & time 	& solution     & rk & time\\	
\hline 
$n=100$ 	& $d=1$ & 4.8496e+01 &2.0&    2.8s	& 	inf		&	- & - \\
		& $d=2$ & 9.6145e+01 &2.2&    1.7s	&  \textbf{9.6197e+01}  & 1 &  2.4s	\\
		& $d=3$ & 9.6184e+01 &2.1&    4.5s	&   -	   	& - &  -	\\ 
		& $d=4$ & 9.6195e+01*&1.3&   18.2s	&   -	   	& - &  -	\\ 
\hline 
$n=200$ 	& $d=1$ & 9.7496e+01 &2.0&    2.7s	& 	inf		&	- & - \\
		& $d=2$ & 1.9512e+02 &2.1&    3.2s	&  \textbf{1.9519e+02} & 1 &  4.6s	\\
		& $d=3$ & 1.9516e+02 &2.1&    5.9s	&   -	   	& - &  -	\\ 
		& $d=4$ & 1.9395e+02*&3  &  565.6s	&   -	   	& - &  -	\\ 
\hline 	
$n=300$ 	& $d=1$ & 1.4650e+02 &2.0&    3.9s	& 	inf		&	- & - \\
		& $d=2$ & 2.9410e+02 &2.1&    4.8s	&  \textbf{2.9418e+02} & 1 &  6.9s	\\
		& $d=3$ & 2.9414e+02 &2.0&    9.3s	&   -	   	& - &  -	\\ 
		& $d=4$ & 2.9176e+02*&3  &  695.6s	&   -	   	& - &  -	\\ 
\hline 	
$n=400$ 	& $d=1$ & 1.9550e+02 &2.0&    5.2s	& 	inf		&	- & - \\
		& $d=2$ & 3.9308e+02 &2.1&    6.5s	&  \textbf{3.9317e+02} & 1 &  9.4s	\\
		& $d=3$ & 3.9312e+02*&2.0&   27.6s	&   -	   	& - &  -	\\ 
		& $d=4$ &-8.1403e+05*&3  &  801.9s	&   -	   	& - &  -	\\ 
\hline 	
$n=500$ 	& $d=1$ & 2.4450e+02 &2.0&    6.8s	& 	inf		&	- & - \\
		& $d=2$ & 4.9206e+02 &2.0&    8.3s	&  \textbf{4.9216e+02} & 1 &  12.4s	\\ 
		& $d=3$ & 4.9210e+02*&2.0&   31.8s	&   -	   	& - &  -	\\ 
		& $d=4$ & 4.9215e+02*&3  & 1144.6s	&   -	   	& - &  -	\\ 
		\hline 	
$n=600$ 	& $d=1$ & 2.9350e+02 &2.0&    8.1s	& 	inf		&	- & - \\
		& $d=2$ & 5.9104e+02 &2.0&   10.4s	& \textbf{5.9115e+02} & 1 &  15.5s	\\
		& $d=3$ & 5.9108e+02*&2.0&   22.3s	&   -	   	& - &  -	\\ 
		& $d=4$ & 5.9114e+02*&1.0&  111.6s	&   -	   	& - &  -	\\ 
\hline 					
\end{tabular}		
\caption{Comparison Sparse-BSOS ($k=2$) and Sparse-PUT on the Generalized Rosenbrock Function}
\label{tab:rosenbrock}
\end{table}

As in the previous examples we find a unique minimizer with Sparse-PUT, certified by the rank condition. This time Sparse-BSOS is not able to find the optimum even when going up to the relaxation $d=4$. However, even though Sparse-BSOS does not obtain the optimal value at an early relaxation, its optimal value at step $d=2$ is already in the right order of magnitude and can be computed faster than the optimal value provided by Sparse-PUT.

Note that for $n= 100$ the relaxation $d=1$ is slower than the relaxation $d=2$. The same happens in Table \ref{tab:broyban} for some values of $n$. This is an issue related to our configuration and could not be reproduced when using another SDP solver or another operating system, respectively. 

To close this section we consider the following test functions of degree $6$:
\begin{itemize}
\item The \textit{Discrete Boundary Value Function}: 
\[
f:= \sum_{i=1}^n(2x_i-x_{i-1}-x_{i+1}+\frac{1}{2}h^2(x_{i} + ih+1)^3)^2,
\]
where $h := \frac{1}{n+1}$,$ x_0:=9=:x_{n+1}$. The sparsity pattern is $\n = (p\times 3)$ and $\oo= 2$.
\item The \textit{Broyden Banded Function}:
\[
f:= \sum_{i=1}^n\left( x_{i}(2 + 10 x_{i}^2) + 1 - \sum_{j \in H_i}(1 + x_{j})x_{j}\right)^2,
\]
where $H_i := \{j : j \neq i, \max(1,i-5)\leq j \leq \min(n,i+1) \}$. The sparsity pattern is $\n = (p\times 7)$ and $\oo= 6$.
\end{itemize}

\textbf{Table \ref{tab:discbound}:} Solving the Discrete Boundary Value Function for $n=15,\ldots,35$. The first possible relaxation for Sparse-PUT is $d=3$. For Sparse-BSOS we choose $k=3$ and compute the first relaxations until we get the certified optimal value. 

\begin{table} 
\begin{tabular}{|l|ccrr||ccrr||} 
\hline 
				&      \multicolumn{4}{c||}{Sparse-BSOS}& \multicolumn{4}{c||}{Sparse-PUT}\\
DiscreteBoundary	&  rel. & solution     	& rk & time 	&  rel. & solution     & rk & time\\	
\hline 
$n=15$ 	& $d=1$ & \textbf{9.8705e-04} 	&1	&  1.4s	& $d=3$& \textbf{9.8705e-04} & 1 &  2.0s	\\
\hline 
$n=20$ 	& $d=1$ & \textbf{4.4893e-04} 	&1  &  1.8s	& $d=3$& \textbf{4.4893e-04} & 1 &  2.6s	\\
\hline 	
$n=25$ 	& $d=1$ & \textbf{2.4060e-04} 	&1	&  2.3s	& $d=3$& \textbf{2.4060e-04} & 1 &  3.3s	\\
\hline 	
$n=30$ 	& $d=1$ & 1.4358e-04 			&2.1&  2.7s	& $d=1$&   inf		& - &  -	\\ 
		& $d=2$ & 1.4359e-04 			&1.1	&  3.2s	& $d=2$&   inf		& - &  -	\\ 
		& $d=3$ & \textbf{1.4358e-04} 	&1  	&  4.0s	& $d=3$& \textbf{1.4359e-04}& 1 &  3.9s	\\
\hline 	
$n=35$ 	& $d=1$ & 9.2438e-05 			&4  	&  3.9s	& $d=1$&   inf		& - &  -	\\ 
		& $d=2$ & 9.2438e-05* 			&3.8	&  4.3s	& $d=2$&   inf		& - &  -	\\ 
		& $d=3$ & \textbf{9.2439e-05} 	&1  	&  4.8s	& $d=3$& \textbf{9.2441e-05}& 1 &  4.5s	\\\hline 					
\end{tabular}		
\caption{Comparison Sparse-BSOS ($k=3$) and Sparse-PUT on the Discrete Boundary Value Function}
\label{tab:discbound}
\end{table}

Note that Sparse-BSOS and Sparse-PUT certify different optimal values in the case $n=30$ and $n=35$ and that in contrast to the theory, the series of lower bounds computed by Sparse-BSOS for $n=35$ is not monotonously increasing. The difference however is less than $10^{-8}$ and can be considered to be zero \enquote{numerically}. As for the Chained Functions in Table \ref{tab:chainwood} \& \ref{tab:chainsing} both Sparse-BSOS and Sparse-PUT are able to certify the minimum in all cases. When Sparse-BSOS succeeds to do so at an early step of the relaxation it is faster and if not then it takes approximately the same time.

\vspace{11pt}
\textbf{Table \ref{tab:broyban}:} Solving the Broyden Banded Function for $n=7,\ldots,15$. The first possible relaxation for Sparse-PUT is $d=3$. For Sparse-BSOS we let $k=3$ and compute the first relaxations.

\begin{table} 
\begin{tabular}{|l|c||crr||crr||crr|} 
\hline 
				&     	& \multicolumn{3}{c||}{Sparse-BSOS}& \multicolumn{3}{c||}{Sparse-PUT}\\
BroydenBanded	&  rel. & solution     & rk & time 	& solution     & rk & time\\	
\hline 
$n=7$ 	& $d=1$ & 2.1371 &2  &  11.5s	&   inf & - &  -	\\
		& $d=2$ & 2.7522 &2  &   9.2s	&   inf & - &  -	\\
		& $d=3$ & 3.1161 &2  &  11.0s	& \textbf{3.4233} & 1 &  15.2s	\\
\hline 
$n=9$ 	& $d=1$ & 2.2171 &3  &  77.0s	&   inf & - &  -	\\
		& $d=2$ & 2.8313 &3  &  72.7s	&   inf & - &  -	\\
		& $d=3$ & 3.1354 &3  &  87.1s	& \textbf{3.3941} & 1 &  105.6s	\\
\hline 	
$n=11$ 	& $d=1$ & 2.2968 &3  & 160.3s	&   inf & - &  -	\\
		& $d=2$ & 3.0108 &2  & 159.7s	&   inf & - &  -	\\
		& $d=3$ & 3.2638 &4  & 190.7s	& \textbf{3.3924 }& 1 &  215.1s	\\
\hline 	
$n=13$ 	& $d=1$ & 2.3353 &3  & 282.9s	&   inf & - &  -	\\
		& $d=2$ & 3.0963 &2.9& 301.6s	&   inf & - &  -	\\
		& $d=3$ & 3.3268	 &4.4& 367.5s	& \textbf{3.4120} & 1 &  357.7s	\\
\hline 	
$n=15$ 	& $d=1$ & 2.3555 &3  & 445.2s	&   inf & - &  -	\\
		& $d=2$ & 3.1514 &3.8& 466.3s	&   inf & - &  -	\\
		& $d=3$ & 3.3617 &3.8& 509.1s	& \textbf{3.4243} & 1 &  545.2s	\\
\hline 					
\end{tabular}
\caption{Comparison Sparse-BSOS ($k=3$) and Sparse-PUT on the Broyden Banded Function}
\label{tab:broyban}
\end{table}

As for the Generalized Rosenbrock Function (Table \ref{tab:rosenbrock}) Sparse-PUT is able to find and certify the optimal solution at the first possible relaxation step, whereas Sparse-BSOS does not succeed to do so in the first three steps. However up to relaxation order $d=3$ Sparse-BSOS is faster than Sparse-PUT and hence provides lower bounds for the objective function in less time and reasonably close to the optimal value. 

\subsubsection{Random medium scale quadratic and quartic test problems}

So far we have compared Sparse-BSOS and Sparse-PUT on examples where the first possible relaxation of Sparse-PUT is exact. We now present examples where this is not the case or the first relaxation cannot even be computed because it is already to large. To this end we choose sparsity patterns with $40$ to $80$ variables in blocks of $10$ to $40$ and fixed overlap $\oo=5$. Note that the crucial parameter for the sparse hierarchies is not so much the total number of variables but rather the maximum block size of the sparsity pattern.

We change Problem \eqref{eq:QP} slightly to generate the following sample of problems. Given a sparsity pattern $I=\{I_1,\ldots,I_p\}$ we now consider:

\begin{equation}\label{eq:QP1}
\min_\x\left\lbrace \sum_{i=1}^na_ix^4 + x^TAx + b^Tx\;:\;  1 - x_i^2\geq 0, \quad x_i \geq 0 \quad i = 1,\ldots,n,\right\rbrace,
\tag{QP'}
\end{equation}
where $b$ is a random vector and the symmetric matrix $A$ is randomly generate according to $I$. 
We verify that $A$ has positive and negative eigenvalues to make sure that our problem is non-convex again. When $a=0\in\R^n$ we refer to \eqref{eq:QP1} as a \textit{quadratic} problem. When mentioning the \textit{quartic} problem (\ref{eq:QP1}), it means that we chose $a$ randomly in $[-1,1]^n$.

\vspace{11pt}
\textbf{Table \ref{tab:QP_quad}:} We consider the quadratic instance of Problem \eqref{eq:QP1} and compute, whenever possible, the first two relaxations of BSOS, Sparse-BSOS and Sparse-PUT.

\begin{sidewaystable}
\begin{tabular}{|c|c|c||c|c|r||c|c|r||c|c|r||}
\hline
\eqref{eq:QP1}		&		&     	& \multicolumn{3}{c||}{BSOS} &  \multicolumn{3}{c||}{Sparse-BSOS}	& \multicolumn{3}{c||}{Sparse-PUT}\\
Quadratic			& $n$ 	&rel. 	& solution     	& rk & time 	& solution     	& rk & time		& solution     	& rk & time 	\\	
\hline
$(7\times10)$		&$ 40$	&$d=1$	& -1.2496e+02 	& 4 &	3.5s		&-1.2496e+02			& 4.0& 	1.8s		&	-1.2496e+02		& 4.0& 	 0.8	s\\
					&		&$d=2$	& -4.4436e+01* 	&15 & 574.4s		&-4.4457e+01			& 3.9& 	8.8s		&\textbf{4.4326e+01}	& 1  & 	45.8s\\	
\hline
$(2\times20,10)$		&$ 40$	&$d=1$	& -1.6197e+02 	& 5 &   3.4s		&-1.6197e+02			& 5.0&	0.7s		&	-1.6197e+02 		& 5.0& 	 0.7	s\\
					&		&$d=2$	& -5.9412e+01*	&16 & 592.9s		&-5.9447e+01			& 9.3&  6.5s		&		-			& -  &	-	\\
\hline
$(15\times10)$		&$ 80$	&$d=1$	& -2.7552e+02* 	&  8&  71.7s		&-2.7557e+02			&4.0 & 	0.8s		&	-2.7557e+02 		& 4.0& 	 0.8	s\\
					&		&$d=2$	&		-		& -	&	-		&-1.0837e+02			&2.4 &	2.7s		&\textbf{1.0825e+02}	& 1  & 143.7s	\\
\hline
$(5\times20)$		&$ 80$	&$d=1$	& -3.4782e+02*	&  5&  86.6s		&-3.4782e+02   		& 5.0&	1.6s		& 	-3.4782e+02 		& 5.0 &	 1.6s\\
					&		&$d=2$	&		-		&  -& -			&-1.2536e+02 		&9.0 & 26.7s		&		-			&	-&	-	\\	
\hline
$(2\times40,10)$		&$ 80$	&$d=1$	& -4.8983e+02*	&  5&  69.3s		&-4.8988e+02 		&5.0 &	3.9s		&	-4.8988e+02 		& 5.0& 	 4.2	s\\
					&		&$d=2$	&		-		& -	&	-		&\textbf{-1.8564e+02}&1	 & 66.8s		&		-			&	-&	-	\\
\hline	
$(23\times10)$		&$120$	&$d=1$	& -3.8765e+02* 	&  8& 849.2s		&-3.8765e+02 		& 4.0& 	0.9s		& 	-3.8765e+02 		& 4.0 &	 1.0	s\\
					&		&$d=2$	&		-		&-	& -			&-1.5884e+02 		& 2.2& 	4.6s		&		-			& -	  &	-	\\
\hline
$(7\times20,15)$		&$120$	&$d=1$	& -5.4921e+02* 	&  5& 772.0s		&-5.4920e+02  		& 4.6& 	3.6s		&	-5.4920e+02 		& 4.6 & 	 3.8	s\\
					&		&$d=2$	&	-			&-	& -			&-2.3846e+02* 		& 6.5& 85.2s		&		-			& -	  &	-	\\
\hline
$(3\times40,15)$		&$120$	&$d=1$	& -7.1721e+02 * 	&  6& 581.0s		&-7.1720e+02 		&6.0 &	9.8s		&	-7.1720e+02 		& 6.0 &	11.1s\\	
					&		&$d=2$	&		-		&-	& -			&-2.3079e+02  		&12.2&143.8s		&		-			& -	  &	-	\\
\hline
\end{tabular}
\caption{ Comparison BSOS($k=1$), Sparse-BSOS($k=1$), and Sparse-PUT on Quadratic Problem \eqref{eq:QP1}, $\oo=5$}
\label{tab:QP_quad}
\end{sidewaystable}

We were able to solve the first relaxation for all algorithms. However, as BSOS cannot benefit from the sparsity structure, it runs into numerical problems, in particular for the examples with $n=120$ variables. Note that in \cite{bsos} problems comparable to this one have been solved by BSOS. There the authors were able to solve the second relaxation for a quadratic problem with $100$ variables. Indeed we were able to compute the second relaxation for some examples with $80$ variables. However, this depends strongly on which SDP constraints BSOS deletes before handing over the system to the solver. We decided not to search for examples where we can compute the second relaxation.

When the solver does not run into numerical problems all solutions of the first relaxations coincide. This is because of the degree of the constraints and the objective function, the first relaxations are more or less the same. In fact the sum of squares weights of the constraints in Sparse-BSOS are all of degree $0$, i.e. they are non-negative scalar variables (and are implemented as such). BSOS and Sparse-BSOS use twice as many constraints because they not only consider the constraints $g_j(x)\geq0$ but also the constraints $g_j(x)\leq1$. This explains why Sparse-PUT is faster for the first relaxation. 

In a number of cases the second relaxation of BSOS and Sparse-PUT could not be implemented because the psd variables become to big for the solver. Sparse-PUT is able to solve the second relaxation in two cases where the block size is $10$ and we could actually certify optimality by the rank condition. The examples with same block size but $n=120$ variables could not be solved because the solver runs out of memory. The same happened for examples with larger block size.

Only Sparse-BSOS was able to compute the second relaxation in all cases. Of course this second relaxation is weaker than the second relaxation of Sparse-PUT and Sparse-BSOS could only certify optimality in one case. However, when comparing the results with the certified values from Sparse-PUT, we see that they are actually quite close and much less time was spent to compute them. In all cases where Sparse-PUT could not solve the second relaxation, Sparse-BSOS could provide a lower bound that is much better that than the one provided by the the first relaxation of Sparse-PUT.

\vspace{11pt}
\textbf{Table \ref{tab:QP_quat}:} The quartic version ($a\neq0$) of Problem \eqref{eq:QP1}. As the degree of the objective function is now $4$, the first relaxation of Sparse-PUT (i.e. with $d=1$) is infeasible. Choosing $k=1$ as fixed parameter, the respective relaxations with $d=1$ of BSOS and Sparse-BSOS are not feasible either, and therefore one computes the second and the third relaxation, whenever possible.

\begin{sidewaystable}
\begin{tabular}{|c|c|c||c|c|l||c|c|l||c|c|l||}
\hline
\eqref{eq:QP1}		&		&     	& \multicolumn{3}{c||}{BSOS} &  \multicolumn{3}{c||}{Sparse-BSOS}	& \multicolumn{3}{c||}{Sparse-PUT}\\
Quartic				& $n$ 	&rel. 	& solution     	& rk & time 	& solution     	& rk & time		& solution     	& rk & time 	\\	
\hline
$(7\times10)	$		&$ 40$	&$d=2$	& -5.4736e+01*	&16	 &466.3s	&-5.4802e+01		&6.0&	  8.9s	&\textbf{-5.2576e+01} 	& 1 & 49.6s	\\	
					&		&$d=3$	&		 - 		& - 	 & - 	&-5.3047e+01*	&3.1&	319.0s	&		 - 		& - 	&	 - 	\\
\hline
$(2\times20,10)$		&$ 40$	&$d=2$	& -6.4481e+01*  &17	 &606.3s	&-6.4528e+01 	&8.7&	  5.9s	&		 - 		& - 	&	 - 	\\
\hline
$(15\times10)$		&$ 80$	&$d=2$	&		 - 		& - 	 &	 - 	&-1.2037e+02 	&3.8&	  3.2s	&\textbf{-1.1897e+02}	& 1 	& 150.1s 	\\
					&		&$d=3$	&		 - 		& - 	 & - 	&-1.1950e+02 	&1.7&	 37.9s	&		 - 		& - 	&	 - 	\\
\hline
$(5\times20)$		&$ 80$	&$d=2$	&		 - 		& - 	 &	 - 	&-1.2725e+02 	&8.4&	 23.5s	&		 -  		& - 	&	 - 	\\	
\hline
$(2\times40,10)$		&$ 80$	&$d=2$	&		 - 		& - 	 &	 - 	&-1.9476e+02 	&12.3&   78.7s	&		 - 		& - 	&	 - 	\\
\hline	
$(23\times10)$		&$120$	&$d=2$	&		 - 		& - 	 &	 - 	&-1.6791e+02 	&4.0	 & 	  8.7s	& 		-		& - 	&	 - 	\\
					&		&$d=3$	&		 - 		& - 	 &	 - 	&-1.6375e+02 	&1.2 &	103.0s	&		 - 		& - 	&	 - 	\\
\hline
$(7\times20,15)$		&$120$	&$d=2$	&		 - 		& - 	 &	 - 	&2.1229e+02 		&9.5	 & 	 54.7s	&		 - 		& - 	&	 - 	\\
\hline
$(3\times40,15)$		&$120$	&$d=2$	&		 - 		& - 	 &	 - 	&-2.3994e+02 	&11.8&	137.1s	&		 - 		& - 	&	 - 	\\
\hline
\end{tabular}
\caption{ Comparison BSOS($k=1$), Sparse-BSOS($k=1$), and Sparse-PUT on Quartic Problem \eqref{eq:QP1}, $\oo=5$}
\label{tab:QP_quat}
\end{sidewaystable}

Sparse-PUT could solve the second relaxation only for the patterns $(7\times10)$ and $(15\times10)$ in which case the optimal solution was attained and certified. 
With $k=1$ (i.e. the SOS in \eqref{aux_{1}_{1}} is of degree at most $2$ and so less than the degree ($=4$) of the objective function) one still may solve higher relaxations with the Sparse-BSOS hierarchy. In contrast to the previous example the first relaxations of Sparse-BSOS and Sparse-PUT do not yield the same values. Indeed the resulting relaxations are actually different because the sum of squares in Sparse-PUT for $d=2$ has degree $4$ whereas in (Sparse-)BSOS it is of degree $2$. Also in the sparse positivity certificate of Sparse-PUT, the SOS weights of the constraints are of degree $2$ whereas in (S)BSOS the weights associated with the products of constraints are non-negative scalars. Note however that the values of the second relaxation of Sparse-BSOS for the patterns $(7\times10)$ and $(15\times10)$ are not so far from being optimal, which suggests that for the other test problems they are not so bad either. In any case Sparse-BSOS is able to provide lower bounds for larger block size, whereas the other hierarchies already overpassed their limit.

\section{Conclusion}
We have provided a sparse version of the BSOS hierarchy \cite{bsos} so as  to handle large scale polynomial optimization problems that satisfy a structured sparsity pattern. The positivity certificates used in the sparse BSOS hierarchy are coming from a sparse version of Krivine-Stengle Positivstellensatz, also proved in this paper. 

We have tested the  Sparse-BSOS hierarchy on a sample  of non-convex problems randomly generated as well as on some typical examples from the literature. The results show that the hierarchy is able to solve small scale dense and large scale sparse polynomial optimization problems in reasonable computational time. In all our experiments where the problem size allowed to compare the dense and sparse versions, finite convergence in the latter took place whenever it took place for the former and moreover at the same relaxation order. This is remarkable, since in principle convergence of the dense version is at least faster than convergence for the sparse version.

In principle the sparse version \cite{waki} of the standard SOS hierarchy is hard to beat as it has proved to be efficient and successful in a number of cases. However in many cases the first relaxation of the Sparse-BSOS hierarchy is (provably) at least as good as the first relaxation of the former
(e.g. for quadratic/quadratically constrained problems), and it can also provide better lower bounds in cases where the former cannot solve the first or second relaxation because the size the semidefinite constraint is till too large for the solver. (In this case we take full advantage of the fact that in Sparse-BSOS relaxations, the size of the semidefinite constraint is chosen and fixed.) We have seen some such examples. This is particularly interesting as it could help solve hard MINLP problems by Branch \& Bound methods where for efficiency one needs to compute good quality lower bounds at each node of the search tree (and one does not need to provide the exact optimal value).

Crucial in our implementation is the comparison of coefficients to state that two polynomials are identical (instead of checking their values on a sample of generic points). This limits the application to problems  with polynomials of small degree (say less than $4$). In particular if some degree in
the problem data is at least $6$ and the block size is not small, the resulting SDP can become ill-conditioned when the relaxation order increases. Depending on the context in which one wants to use the sparse hierarchy, an alternative may be to implement polynomial identities by sampling.
	
\section*{Acknowledgement}
The first author is very thankful to the National University of Singapore and Professor Kim-Chuan Toh for their hospitality and financial support during his stay in Singapore. 

\end{document}